\documentclass[12pt]{amsart}
\usepackage{amsmath,amssymb,amsfonts,amsthm,amsopn}
\usepackage{graphics}
\usepackage{mathrsfs}
\usepackage{graphicx,tikz}
\usepackage[all]{xy}
\usepackage{amsmath}
\usepackage{multirow, longtable, makecell, caption, array}
\usepackage{amssymb}
\usepackage{mathtools}
\usepackage{pb-diagram}
\usepackage{cellspace} %

 \def\Spnr{Sp(d,\R)}
 \def\Gltwonr{GL(2d,\R)}

\newcommand{\tfs}{time-frequency shift}

\newcommand{\modsp}{modulation space}

\newtheorem{tm}{Theorem}[section]
\newtheorem{lemma}[tm]{Lemma}
\newtheorem{prop}[tm]{Proposition}
\newtheorem{cor}[tm]{Corollary}

\newtheorem{theorem}{Theorem}[section]
\newtheorem{corollary}[theorem]{Corollary}
\newtheorem{definition}[theorem]{Definition}
\newtheorem{example}[theorem]{Example}

\newtheorem{proposition}[theorem]{Proposition}
\newtheorem{remark}[theorem]{Remark}

\newcommand{\beqa}{\begin{eqnarray*}}
\newcommand{\eeqa}{\end{eqnarray*}}

\newcommand{\field}[1]{\mathbb{#1}}
\newcommand{\bR}{\field{R}}    
\newcommand{\bN}{\field{N}}    
\newcommand{\bZ}{\field{Z}}    
\newcommand{\bT}{\field{T}}    %
\def\la{\lambda}

\def\eps{\epsilon}

 \def\cF{\mathcal{F}}       
 \def\cS{\mathcal{S}}
 \def\cD{\mathcal{D}}
 
 \def\cB{\mathcal{B}}
 
 \def\cG{\mathcal{G}}
 \def\cM{\mathcal{M}}

 \def\cA{\mathcal{A}}
 
 \def\cI{\mathcal{I}}
 \def\cC{\mathcal{C}}
 \def\cW{\mathcal{W}}

\def\rd{\bR^d}

\def\rdd{{\bR^{2d}}}
\def\zdd{{\bZ^{2d}}}

\def\lrd{L^2(\rd)}
\def\lrdd{L^2(\rdd)}

\def\intrd{\int_{\rd}}
\def\intrdd{\int_{\rdd}}

\def\R{\right)}

\def\<{\left<}
\def\>{\right>}

\def\inv{^{-1}}

\def\mv1{M_v^1}

\def\phas{(x,\xi )}

\def\o{\xi}

\def\R{\mathbb{R}}
\def\Ren{\mathbb{R}^d}

\def\Opw{Op_{w}}
\def\Sn2{S_{2}(L^{2}(\Ren))}
\def\S1{S_{1}(L^{2}(\Ren))}
\def\sig00{\sigma_{0,0}}

\def\la{\langle}
\def\ra{\rangle}

\def\spdr{{\mathfrak {sp}}(d,\R)}

\newcommand{\A}{\mathcal{A}}

\begin{document}

\title[Wigner Analysis of
Operators]{Wigner Analysis of
Operators.\\ Part II: Schr\"{o}dinger equations}

\author{Elena Cordero, Gianluca Giacchi and Luigi Rodino}
\address{Department of Mathematics, University of Torino, Italy}
\address{Dipartimento di Matematica, Università di Bologna, Italy; University of Lausanne, Switzerland; Centre Hospitalier Universitaire Vadois, Switzerland; Haute Ecole Spécialisée de Suisse Occidentale, Switzerland}
\address{Dipartimento di Matematica, University of Torino, Italy}
\email{elena.cordero@unito.it}
\email{gianluca.giacchi2@unibo.it}
\email{luigi.rodino@unito.it}

\date{}

\begin{abstract}
	We study the phase-space concentration of the so-called \emph{generalized metaplectic operators} whose main examples are Schr\"odinger equations with bounded perturbations. 
	
	To reach this goal, we perform a so-called \emph{$\cA$-Wigner analysis} of the previous equations, as started in Part I, cf. \cite{CR2021}. Namely, the classical Wigner distribution is extended by considering a class of time-frequency representations constructed as images of metaplectic operators acting on symplectic matrices $\cA\in Sp(2d,\bR)$. Sub-classes of these representations, related to  covariant symplectic matrices, reveal to be particularly suited for the time-frequency study of the Schr\"odinger evolution. This testifies the effectiveness of this  approach for such equations, highlighted by the development of a related wave front set. 
	
	We first study the properties of $\cA$-Wigner representations and related pseudodifferential operators needed for our goal. This approach paves the way to new quantization procedures.

	As a byproduct, we introduce new   quasi-algebras of generalized metaplectic operators containing Schr\"odinger equations with more general potentials, extending the results contained in the previous works  \cite{CGNRJMPA,CGNRJMP2014}.

\end{abstract}

\maketitle

\section{Introduction}

Cauchy problems for Schr\"odinger
equations  have been studied by a variety of authors in many different frameworks. Limiting attention to the Microlocal Analysis context, let us mention as a partial list of contributions \cite{asada-fuji,Craig1995,HW2005,Ito2006,Ito2009,Nakamura2005,NR2015,Robbiano1999,Wunsch1999}.

As more recent issues, under the influence of the new Time-Frequency methods, we may refer to \cite{KB2020,CarypisWahlberg21,CGNRJMPA,CNR2015,CNRadv2015,CR2022,Elena-book,CR2021,DdGP2013,Birkbis,deGossonQHA21,KKI2012,PRW2018,P2018,W2018}.

Here we propose  a new approach, in terms of phase-space concentration of suitable time-frequency distributions. The basic idea in terms of Wigner distribution is not new, though. It goes back to Wigner  1932 \cite{Wigner32} 
 (later  developed by Cohen and  many other authors, see e.g. \cite{Cohen1,Cohen2}).
 \begin{definition} Consider $f,g\in\lrd$. 
 	The cross-Wigner distribution $W(f,g)$ is
 	\begin{equation}\label{CWD}
 	W(f,g)\phas=\intrd f(x+\frac t2)\overline{g(x-\frac t2)}e^{-2\pi i t\o}\,dt.
 	\end{equation} If $f=g$ we write $Wf:=W(f,f)$, the so-called Wigner distribution of $f$.
 \end{definition}
 For a given linear operator $P$ acting on $\lrd$ (or a more general functional space), Wigner considered an operator $K$ on $\lrdd$ such that 
\begin{equation}\label{I3}
W(Pf) = KW(f)
\end{equation}
and its kernel $k$
\begin{equation}\label{I4}
W(Pf)\phas = \intrdd k(x,\xi,y,\eta) Wf(y,\eta)\,dyd\eta.
\end{equation}

We continue the development of a theory started in the Part I \cite{CR2021}, addressed to $P$ pseudodifferential operators with $W$ replaced by the more general $\tau$-Wigner distributions. Here the  main concern is the study of Cauchy propagators for linear Schr\"odinger
equations 
\begin{equation}\label{C1}
\begin{cases} i \displaystyle\frac{\partial
	u}{\partial t} +H u=0\\
u(0,x)=u_0(x),
\end{cases}
\end{equation}
with $t\in \bR$ and the initial condition $u_0\in\cS(\rd)$ (Schwartz class) or in 
some modulation space as explained below. The Hamiltonian has the form
\begin{equation}\label{C1bis}
H=Op_w(a)+ Op_w(\sigma),
\end{equation}
where $Op_w(a)$ is the Weyl quantization of a real homogeneous quadratic polynomial on
$\rdd$ and $Op_w(\sigma)$ is a pseudodifferential operator with a
symbol $\sigma$ in suitable modulation spaces, namely $\sigma\in M^{\infty,q}_{1\otimes v_s}(\rdd)$, $s\geq 0$, $0<q\leq 1$ (see Section 2.2 below for the definitions) which guarantee that $Op_{w}(\sigma)$ is bounded
on $\lrd $ (and in more general spaces).  This implies that the operator $H$ in \eqref{C1bis} is a bounded perturbation of the generator $H_0 = Op_w(a)$ of a unitary
group (cf. \cite{RS75} for details).

As special instances of the Hamiltonian above we find the Schr\"odinger equation $H= \Delta - V(x)$ and the perturbation of the
harmonic oscillator $H= \Delta - |x|^2 - V(x)$ with a potential $V\in
M^{\infty,q} (\rd )$. Observe that $V$ is bounded, but not necessarily smooth.

The unperturbed case $\sigma=0$, was already considered in \cite{CR2022}
\begin{equation}\label{C12}
\begin{cases} i \displaystyle\frac{\partial
	u}{\partial t} +Op_w(a) u=0\\
u(0,x)=u_0(x).
\end{cases}
\end{equation}
The solution is given by the metaplectic operators
$u=\mu(\chi_t)u_0$, for a suitable symplectic matrix
$\chi_t$, see for example the textbooks \cite{folland,Birkbis}. Precisely, if $a(x,\xi)=\frac{1}{2}xAx+\xi B
x+\frac{1}{2}\xi C\xi$, with $A, C$ symmetric and $B$ invertible, we
can consider the classical evolution, given by the linear Hamiltonian
system
\[
\begin{cases}
2\pi \dot x=\nabla _\xi a =Bx+C\xi\\
2\pi \dot \xi=-\nabla _x a =-A x-B^T\xi
\end{cases}
\]
(the factor $2\pi$ is due to our normalization of the Fourier transform) with Hamiltonian matrix $\mathbb{D}:=\begin{pmatrix}B&C\\
-A&-B^T\end{pmatrix}\in \mathrm{sp} (d,\bR )$. Then we have
$\chi_t=e^{t\mathbb{D}}\in Sp(d,\R)$.

The solution to \eqref{C12} is  the Schr\"odinger propagator \begin{equation}\label{e7}u(t,x)=e^{it\Opw(a)}u_0(x)=\mu(\chi_t)u_0,\end{equation}
and the Wigner transform with respect to the space variable $x$ is given by
$$Wu(t,z)=W u_0(\chi_t^{-1}z), \quad z=\phas,$$
as already observed in the works of Wigner \cite{Wigner32} and Moyal-Bartlett \cite{BM49}. Hence \eqref{I4} reads in this case
\begin{equation}\label{3bis}
W(e^{it Op_w(a)}u_0)(z)=\int_{\rdd} k(t,z,w) Wu_0(w)\,dw,
\end{equation}
with $k(t,z,w)$ given by the delta density $\delta_{z=\chi_tw}$. \par The aim of \cite{CR2022} was  to reconsider \eqref{e7} and \eqref{3bis} in the functional frame of the modulation spaces, in terms of the general $\cA$-Wigner transform introduced in \cite{CR2021}, see Definition \ref{def4.1} below.
The propagator of the  perturbed problem 
\eqref{C1bis} 
is a generalized metaplectic operator, as already exhibited in Theorem 4.1 \cite{CGNRJMP2014} for symbols in the Sj\"ostrand class. 

Here, to deal with further non-smooth potentials $Op_w(\sigma)$ in \eqref{C1bis}, $\sigma\in M^{\infty,q}_{1\otimes v_s}(\rdd)$,  we enlarge the class of generalized metaplectic operators, including quasi-algebras of operators, which allow better decay at infinity than the original Sj\"ostrand class.  
To quantify the decay we use the Wiener amalgam spaces $W(C,L^p_{v_s})(\rdd)$, which  consist of the continuous functions $F$ on $\rdd$ such that
\begin{equation}\label{Wiener-space}
\|F\|_{W(C,L^p)}:=\left(\sum_{k\in\zdd}(\sup_{z\in [0,1]^{2d}}|F(z+k)|)^pv_s(k)^p\right)^{\frac1p}<\infty
\end{equation}
(obvious changes for $p=\infty$),
where $v_s(k)=(1+|k|^2)^{1/2}$.

\begin{definition}\label{def1.1} Given $\chi \in \Spnr $,
	$g\in\cS(\rd)$, $0< q\leq 1$, we say that a
	linear operator $T:\cS(\rd)\to\cS'(\rd)$ is a \textbf{generalized metaplectic operator} in the
	class $FIO(\chi,q,v_s)$ if there exists a function $H\in W(C,L^q_{v_s})(\rdd)$, 
	such that the kernel of $T$ with respect to \tfs s satisfies the decay
	condition
	\begin{equation}\label{asterisco}
	|\langle T \pi(z) g,\pi(w)g\rangle|\leq H(w-\chi z),\qquad \forall w,z\in\rdd.
	\end{equation}
\end{definition}

We infer boundedness, quasi-algebras and spectral properties of the previous operators, see Section 6 below. Moreover, we shall show that they can be represented as
$$T=Op_w(\sigma_1)\mu(\chi)\quad \mbox{or}\quad T=\mu(\chi)Op_w(\sigma_2),$$
that is, they can be viewed as composition of metaplectic operators with Weyl operators with symbols in the modulation spaces $M^{\infty,q}_{1\otimes v_s}(\rdd)$.

The solution $e^{itH}$ to \eqref{C1} is a generalized metaplectic operator of this type  for every $t\in\bR$, so that it enjoys the phase-space concentration of this class.

The main work of this paper relies in preparing all the instruments we need to study the Wigner kernel
 of $e^{itH}$, namely $k(t,z,w)$, $w,z\in\rdd$, such that
$$ W(e^{itH}u_0)(z)=\int k(t,z,w) Wu_0(w)\,dw$$
and possible generalizations to $\A$-Wigner distributions, defined as follows. 

{\begin{definition}\label{def4.1}
		Let $\cA\in Sp(2d,\bR)$ be a $4d\times 4d$ symplectic matrix. We
		define the \textbf{metaplectic Wigner distribution} associated to $\cA$ by 
		\begin{equation}\label{WignerA}
		W_\cA (f,g)=\mu(\cA) (f\otimes \bar{g}),\quad f,g\in\lrd,
		\end{equation}
		and set $W_\cA f:= W_\cA (f,f)$.
	\end{definition}
	When the context requires to stress the symplectic matrix $\cA$, that defines the metaplectic Wigner distribution $W_\cA$, we refer to $W_\cA$ as to \textbf{$\cA$-Wigner distribution}.}

We shall focus on  $\cA$   shift-invertible, covariant symplectic matrices, see Definitions 4.5 and Subsection 4.1 and 4.2 below for  definitions and properties. 
Furthermore we limit to  $\cA$   shift-invertible, covariant symplectic matrices such that
the related metaplectic Wigner distribution $W_\cA$ is in the Cohen class $Q_{\Sigma}$, namely it can be written as 
$$W_\cA (f,g)= W(f,g)\ast \Sigma_\cA$$
where the kernel $\Sigma_\cA$ is related to $\cA$ by \eqref{CohenSymbol}, \eqref{covBCohen} below.

Let us define 
$\Sigma_{\cA,t}(z)=\Sigma_\cA(\chi_t(z))$ and denote by $\cA_t$ the  covariant matrix such that
$$W_{\cA_t}=Wf\ast \Sigma_{\cA,t}.$$
Then from the results of \cite{CR2022} we have from the unperturbed equation \eqref{C12}, as counterpart of \eqref{3bis} 
$$
W_\cA(e^{itOp_w(a)}u_0)(z)=\int_{\rdd} \delta_{z=\chi_t w} (W_{\cA_t}u_0)(w)\,dw.
$$
So we keep the action of the classical Hamiltonian flow according to the original idea of Wigner \cite{Wigner32}, provided the matrix $\cA_t$ is defined as before.\par
We prove that the result does not change so much for the perturbed equation. Namely, under the stronger assumption $\sigma\in S^{0}_{0,0}(\rdd)$, we prove (see Proposition \ref{Y} below)
$$
	W_\cA(e^{itH}u_0)(z)=\int_{\rdd} k_{\cA} (t,z,w) (W_{\cA_t}u_0)(w)\,dw
	$$
	where, for every $N\geq0$,
	$$k_{\cA} (t,z,w)\la z-\chi_t(w)\ra ^{2N}$$
	is the kernel od an operator bounded on $\lrd$.

Starting from this, we may obtain the propagation result for the Wigner wave front set
$$\mathcal{W}\cF_{\cA}(e^{itH}u_0)=\chi_t(\mathcal{W}\cF_{\cA}u_0),$$
see Definition \ref{YYY} in the sequel. In particular, for $W_\cA=W=W_t$, defining as in \cite{CR2021,CR2022} $z_0\notin \mathcal{W}\cF f$, $z_0\not=0$, if there exists a conic neighbourhood $\Gamma_{z_0}\subset \rdd$ of $z_0$ such that for all $N\geq0$, 
$$\int_{\Gamma_{z_0}}\la z\ra^{2N} |Wf(z)|^2\,dz<\infty,$$
we obtain 
$$\mathcal{W}\cF(e^{itH}u_0)=\chi_t(\mathcal{W}\cF u_0).$$


The outline of this article is as follows. In Section 2 we establish some background
and notation. In Section 3 we present the main properties of metaplectic Wigner distributions and introduce their related pseudodifferential operators. Different symplectic matrices give rise to different quantizations: we show the link between different quantizations (see Lemma \ref{lemmaComm} below) and generalize the equality in \eqref{I3} to any $\cA$-Wigner distribution and $\cA$-pseudodifferential operator. This is a valuable result of its own, we believe it could be useful in the framework of operator theory and quantum mechanics.  Section $4$ is devoted to study subclasses of $\cA$-Wigner distributions and pseudodifferential operators: covariant, totally Wigner-decomposable and Wigner-decomposable. The last ones provide a new characterization of modulation spaces (cf. Theorem \ref{thm28revised} below).
Next, we show that covariant matrices belong to the Cohen class (Theorem \ref{ThmCohen}) and compute the related kernel. As for the Wigner case, we are able to give an explicit expression of the $\cA$-Wigner when $\cA$ is covariant (see Theorem \ref{charWAcovariant}).
Section $5$ contains a deep study of $\cA$-pseudodifferential operators on modulation spaces, which will be used in the applications to Schr\"odinger equations (Section $7$). Section $6$ introduces new algebras of generalized metaplectic operators and their main properties. Finally, Section $7$ exhibits an application of the theory developed so far to  Schr\"odinger equations.

\section{Preliminaries and notation}

\subsection{Test functions, tempered distributions, Fourier transform}
We denote with $\mathcal{S}(\mathbb{R}^d)$ the space of Schwartz functions and with $\mathcal{S}'(\mathbb{R}^d)$ the space of the tempered distributions, with vector topologies given respectively by the topology of the seminorms of $\mathcal{S}(\mathbb{R})$ and the weak-$\ast$ topology.

 We write $\langle \cdot,\cdot\rangle$ for the unique extension to $\mathcal{S}'(\mathbb{R}^d)\times \mathcal{S}(\mathbb{R}^d)$ of the sesquilinear inner product of $L^2(\mathbb{R}^d)$, namely
\[
	\langle f,g\rangle=\int_{\mathbb{R}^d}f(t)\overline{g(t)}dt \ \ \ \ \ f,g\in L^2(\mathbb{R}^d).
\]
For all $p\in(0,+\infty]$, one has $\mathcal{S}(\mathbb{R}^d)\hookrightarrow L^p(\mathbb{R}^d)\hookrightarrow \mathcal{S}'(\mathbb{R}^d)$ and if $p\neq\infty$, $\mathcal{S}(\mathbb{R}^d)$ is dense in $L^p(\mathbb{R}^d)$.

If $f,g$ are complex-valued Lebesgue-measurable functions on $\mathbb{R}^d$, we denote with $f\otimes g$ the function
\[
	(f\otimes g)(x,y)=f(x)g(y),\quad x,y\in\mathbb{R}^d.
\]


The Fourier transform of a function $f\in\mathcal{S}(\mathbb{R}^d)$ is defined as
\begin{equation}\label{defF}
	\hat{f}(\xi)=\int_{\mathbb{R}^d}f(x)e^{-2\pi i\xi\cdot t}dt \ \ \ \ t\in\mathbb{R}^d,
\end{equation}
where $\xi\cdot t$ denotes the real canonical inner product of $\mathbb{R}^d$. We name $\mathcal{F}$ the Fourier transform operator, mapping $f\in\mathcal{S}(\mathbb{R})$ into $\hat f$, which is a surjective isomorphism of $\mathcal{S}(\mathbb{R})$ into itself with inverse $\cF^{-1}$. It defines a unitary operator on $L^2(\mathbb{R}^d)$:
\[
	\langle f,g\rangle=\langle \hat f,\hat g\rangle, \ \ \ \ \  f,g\in L^2(\mathbb{R}^d)
\]
and, in particular, $\Vert f\Vert_2=\Vert \hat f\Vert_2$, where $\Vert\cdot\Vert_p$ denotes the $L^p$ (quasi-)norm of $L^p(\mathbb{R}^d)$, $0<p\leq\infty$.
If $f\in\mathcal{S}'(\mathbb{R}^d)$, the Fourier transform of $f$ is defined as the tempered distribution $\hat f$ such that
\[
	\langle \hat f,\varphi\rangle=\langle f,\mathcal{F}^{-1}\varphi\rangle, \ \ \ \ \ \forall\varphi\in\mathcal{S}(\mathbb{R}).
\]
If $\Phi\in\mathcal{S}(\mathbb{R}^{2d})$, we define the \textit{partial Fourier transform $\mathcal{F}_2$ of $\Phi$} w.r.t. the second variable as 
\begin{equation}\label{F2}
	\mathcal{F}_2\Phi(x,\xi)=\int_{\mathbb{R}^d}\Phi(x,y)e^{-2\pi i\xi\cdot y}dy.
\end{equation}
The operator $\mathcal{F}_2$ on $\mathcal{S}'(\mathbb{R}^{2d})$ is defined, by density, as 
\[
	\langle\mathcal{F}_2(f\otimes g),\Phi\rangle=\langle f\otimes g,\mathcal{F}_2^{-1}\Phi\rangle, \ \ \ \ \ \  \Phi\in\mathcal{S}(\mathbb{R}^{2d}).
\] 
\subsection{Short-time Fourier transform and modulation spaces}
In this paper $v$ is a continuous, positive, submultiplicative weight function on $\rd$, i.e., 
$ v(z_1+z_2)\leq v(z_1)v(z_2)$, for all $ z_1,z_2\in\Ren$.
A weight function $m$ is in $\mathcal{M}_v(\rd)$ if $m$ is a positive, continuous weight function on $\Ren$ and it is {\it
	$v$-moderate}:
$ m(z_1+z_2)\lesssim v(z_1)m(z_2)$. This notation means that there exists a universal constant $C>0$ such that the inequality $m(z_1+z_2)\leq C v(z_1)m(z_2)$ holds for all $z_1,z_2\in\Ren$. 

 In the following, we will work with weights on $\rdd$ of the type
\begin{equation}\label{weightvs}
v_s(z)=\la z\ra^s=(1+|z|^2)^{s/2},\quad z\in\rdd,
\end{equation}

For $s<0$, $v_s$ is $v_{|s|}$-moderate.\par 
In particular, we shall use the weight functions on $\bR^{4d}$:
\begin{equation}\label{weight1tensorvs}
( v_s\otimes 1)(z,\zeta)=(1+|z|^2)^{s/2},\quad (1\otimes v_s)(z,\zeta)=(1+|\zeta|^2)^{s/2},\quad z,\zeta\in\rdd.
\end{equation}

For a fixed $g\in\mathcal{S}(\mathbb{R}^d)\setminus\{0\}$, the \textbf{short-time Fourier transform} (STFT) of $f\in L^2(\mathbb{R}^d)$ is defined  as
\begin{equation}\label{STFT}
	V_gf(x,\xi)=\int_{\mathbb{R}^d}f(t)\overline{g(t-x)}e^{-2\pi i\xi\cdot t}dt,\quad x,\xi\in\mathbb{R}^d.
\end{equation}
For $x,\xi\in\mathbb{R}^d$, we denote with $T_x$ and $M_\xi$ the translation and the modulation operators respectively, i.e. the unitary operators on $L^2(\mathbb{R}^d)$ defined as
\begin{equation*}
	T_xg(t)=g(t-x),\quad 
	 M_\xi g(t)=e^{2\pi i\xi\cdot t}g(t),\quad g\in L^2(\mathbb{R}^d).
\end{equation*} 
If $z=(x,\xi)\in\mathbb{R}^{2d}$, the operator $\pi(z)=M_\xi T_x$ is called \textbf{time-frequency shift}.
The definition of STFT can be extended to all tempered distributions: fixed $g\in\mathcal{S}(\rd)\setminus\{0\}$,  $f\in\mathcal{S}'(\mathbb{R}^d)$ 
\[
	V_gf(x,\xi)=\langle f,\pi(x,\xi)g\rangle.
\]

The \emph{modulation spaces}, introduced by Feichtinger in \cite{feichtinger-modulation} and extended to the quasi-Banach setting by Galperin and Samarah \cite{Galperin2004}, are now available in many textbooks, see e.g. \cite{KB2020,Elena-book,grochenig}. We recall their definition and main properties.

	Fix a non-zero window $g$ in the Schwartz class $\cS(\rd)$. Consider a weight function $m\in\mathcal{M}_v(\rdd)$ and indices $0<p,q\leq \infty$. The modulation space $M^{p,q}_m(\rd)$ is the subspace of tempered distributions $f\in\cS'(\rd)$ with
\begin{equation}\label{norm-mod}
\|f\|_{M^{p,q}_m}=\|V_gf\|_{L^{p,q}_m}=\left(\intrd\left(\intrd |V_g f \phas|^p m\phas^p dx \right)^{\frac qp}d\o\right)^\frac1q <\infty
\end{equation}
(natural changes with $p=\infty$ or $q=\infty)$. 
We write $M^p_m(\rd)$ for $M^{p,p}_m(\rd)$ and $M^{p,q}(\rd)$ if $m\equiv 1$. 
The space $M^{p,q}_m(\mathbb{R}^d)$ is a (quasi-)Banach space whose definition is
independent of the choice of the window $g$: \emph{different
	non-zero window functions in $\cS(\rd)$ yield equivalent (quasi)-norms}. 
For $1\leq p,q\leq \infty$, they are Banach spaces and the window class can be extended to the modulation space $M^1_v(\rd)$ (Feichtinger algebra). The modulation space $M^{\infty,1}(\rd)$ coincide with the Sj\"ostrand's class in \cite{Sjostrand1}.\par
We recall their inclusion properties: 
\begin{equation}
\mathcal{S}(\mathbb{R}^{d})\subseteq M^{p_{1},q_{1}}_m(\mathbb{R}%
^{d})\subseteq M^{p_{2},q_{2}}_m(\mathbb{R}^{d})\subseteq \mathcal{S}^{\prime
}(\mathbb{R}^{d}),\quad p_{1}\leq p_{2},\,\,q_{1}\leq q_{2}.
\label{modspaceincl1}
\end{equation}%

If $m\in \mathcal{M}_{v_s}$, denoting by $%
\mathcal{M}_m^{p,q}(\mathbb{R}^d)$ the closure of $\mathcal{S}(\mathbb{R}^d)$ in the $M^{p,q}_m$-norm, we observe
\begin{equation*}
\mathcal{M}_m^{p,q}(\mathbb{R}^d) \subseteq M^{p,q}_m(\mathbb{R}^d), \quad 0<p,q\leq\infty,
\end{equation*}
 and 
 \begin{equation*} \mathcal{M}^{p,q}_m (\mathbb{R}^d) = M^{p,q}_m (\mathbb{R}^d),\quad 0<p,q<\infty.
\end{equation*}

\subsection{The symplectic group $Sp(d,\mathbb{R})$ and the metaplectic operators}
We recall definitions and properties of symplectic matrices and metaplectic operators  in a nutshell, referring to \cite{Birkbis} for  details.\par
The standard symplectic matrix $J$ is
\begin{equation}\label{J}
J=\begin{pmatrix} 0_{d\times d}&I_{d\times d}\\-I_{d\times d}&0_{d\times d}\end{pmatrix}
\end{equation}
 and the symplectic group can be defined as
 \begin{equation}\label{defsymplectic}
 \Spnr=\left\{\cA\in\Gltwonr:\;\cA^T J\cA=J\right\},
 \end{equation}
 where $\Gltwonr$ is the group of $2d\times 2d$ real invertible matrices, $\cA^T$ is the transpose of $\cA$ and the  matrix
 $J$ is defined in \eqref{J}.\par
 The symplectic  algebra
 $\spdr$ is the set of
 $2d\times 2d$ real matrices
 $\A$   such that $e^{t \A}
 \in \Spnr$ for all
 $t\in\R$.
 Recall that the metaplectic representation
 $\mu$ is a unitary
 representation of the double
 cover of $\Spnr$ on $\lrd$.
 For elements of $\Spnr$ of special form we can compute the metaplectic
 representation explicitly.  Precisely,  for
 $f\in L^2(\R^d)$,  $C$ real symmetric $d\times d$ matrix ($C^T=C$) we consider the symplectic matrix 
 \begin{equation}\label{Vc}
 V_C=\begin{pmatrix}
 I_{d\times d} & 0_{d\times d}\\
 C & I_{d\times d} 
 \end{pmatrix};
 \end{equation}
 then, up to a phase factor,
 \begin{equation}\label{muVc}
 \mu(V_C)f(t)=e^{i\pi Ct\cdot t}f(t)
 \end{equation}
 for all $f\in L^2(\rd)$. Hence the previous operator is a multiplication by the \emph{chirp}
 \begin{equation}\label{chirp}
\Phi_C(t)= e^{i\pi Ct\cdot t},\quad t\in\rd.
 \end{equation}
 For the standard matrix $J$ in \eqref{J},
 \begin{equation}
 \mu(J)f=\cF f\label{muJ}.
 \end{equation}
 For any $L\in GL(d,\bR)$, we can define the symplectic matrix 
\begin{equation}\label{MotL}
\cD_L= \left(\begin{array}{cc}
L^{-1} &0_{d\times d}\\
0_{d\times d} & L^T
\end{array}\right)  \in Sp(d,\bR),
\end{equation}
 and,  up to a phase factor, 
 \begin{equation}\label{AdL}
 \mu(\cD_L)F(t)=\sqrt{|\det L|}F(Lt)=\mathfrak{T}_{L} F(t),\quad F\in\lrd.
 \end{equation}
 These operators are also called \emph{rescaling operators}.  
 The metaplectic operators enjoy a group structure with respect to the composition.
 \begin{proposition}\label{deGosson96}
 	The metaplectic group is generated by the operators $\mu(J), \mu(\cD_L)$ and $\mu(V_C)$.
 \end{proposition}
The following issue will be used in the sequel.
\begin{lemma}\label{lemmaUC}
Consider the symplectic matrix $V_C$ in \eqref{Vc}. Then,
 \begin{equation}\label{Vcinvtransp}
V_C^{-T}=(V_C^{-1})^T=\begin{pmatrix}
I_{d\times d} & -C\\
0_{d\times d} & I_{d\times d} 
\end{pmatrix}
\end{equation}
and, up to a phase factor, the metaplectic operator $\mu(V_C^{-T})$ is a convolution operator:
\begin{equation}\label{vcT}
\mu(V_C^{-T})f=\cF(\Phi_{C})\ast f,\quad f\in\lrd,
\end{equation}
with $\Phi_{C}$ being the chirp function in \eqref{chirp}. In particular, if the symmetric matrix $C$ is invertible, then
\begin{equation}\label{vcT-inv}
\mu(V_C^T)f=|\det C|\, (\Phi_{-C^{-1}}\ast f),\quad f\in\lrd.
\end{equation} 
\end{lemma}
\begin{proof} Formula \eqref{Vcinvtransp} is an easy computation.
For any $f\in\lrd$, using the definition \eqref{defsymplectic} we can write
	\begin{align*}
	\mu(V_C^{-T})f&=\mu(J V_C J^{-1})=\mu(J)\mu(V_C)\mu(J^{-1})f\\
	&=\cF(\Phi_C \cdot \cF^{-1}f)=\cF(\Phi_C)\ast f.
	\end{align*}
	For $C$ invertible, up to a phase factor, $\cF \Phi_C=|\det C| \Phi_{-C^{-1}}$, see e.g.  \cite{folland}, so that we obtain formula \eqref{vcT-inv}.
\end{proof}

In particular, observe that $V_C^T$ is obtained from $V_C^{-T}$ in \eqref{Vcinvtransp} by replacing the matrix $C$ with $-C$, so that 
\begin{equation}\label{muVctrans}
\mu(V_C^{T})f=\cF(\Phi_{-C})\ast f,\quad f\in\lrd,
\end{equation}
and, if $C$ is invertible,
\begin{equation}\label{muVctrans-inv}
\mu(V_C^{T})f=|\det C|\, (\Phi_{C^{-1}}\ast f),\quad f\in\lrd.
\end{equation}
 This paper deals both with the symplectic group  $Sp(d,\bR)$ of  $2d\times 2d$  matrices and  $Sp(2d,\bR)$ of  $4d\times 4d$ ones. To avoid confusions, in what follows the  matrix $\cA$ denotes a symplectic matrix in $Sp(2d,\bR)$ whereas $\chi$ a matrix in $Sp(d,\bR)$.
\begin{definition}\label{FreeSympMatrices}
	A matrix $\cA\in Sp(2d,\mathbb{R})$, with block decomposition
	\begin{equation}\label{blockAfree}
		\cA=\begin{pmatrix}
			A & B\\  C & D
		\end{pmatrix}
	\end{equation}
	 is a \textbf{free symplectic matrix} if $\det B\neq0$.
\end{definition} 

In this work, free symplectic matrices will be important for two main reasons. For the following Lemma, we refer to \cite[Theorem 4.53]{folland} and \cite[Theorem 60]{Birkbis}.

\begin{lemma}\label{lemmaFreeSymp}
	Let $\cA\in Sp(2d,\mathbb{R})$. Then,\\
	(i) there exist $\cA_1,\cA_2\in Sp(2d,\mathbb{R})$ free such that $\cA=\cA_1\cA_2$;\\
	(ii) if $\cA$ is free with block decomposition (\ref{blockAfree}) then, for every $F\in\mathcal{S}(\rdd)$,
	\begin{equation}\label{muFree}
		\mu(\cA)F(x)=(\det(B))^{-1/2}e^{-i\pi DB^{-1}x\cdot x}\int_{\rdd}F(y)e^{2\pi i (B^{-1}x\cdot y-\frac{1}{2}B^{-1}Ay\cdot y)}dy.
	\end{equation}
\end{lemma}

\section{Metaplectic Pseudodifferential Operators}
We recall  some basic examples and properties of the metaplectic Wigner distribution, for detail see cf. \cite{CR2021}.

\begin{example}\label{es1}
	For $\tau\in[0,1]$, the matrix of $Sp(2d,\mathbb{R})$
	\[
		\cA_{\tau,2d}:=\begin{pmatrix}
		(1-\tau)I_{d\times d} & \tau I_{d\times d} & 0_{d\times d} & 0_{d\times d}\\
		0_{d\times d} & 0_{d\times d} & \tau I_{d\times d} & -(1-\tau)I_{d\times d}\\
		0_{d\times d} & 0_{d\times d} & I_{d\times d} & I_{d\times d}\\
		-I_{d\times d} & I_{d\times d} & 0_{d\times d} & 0_{d\times d}
		\end{pmatrix},
	\]
	defines the $\tau$-Wigner distribution on $\mathcal{S}(\rd)\times\mathcal{S}(\rd)$, that is
	\[
		W_\tau(f,g)(x,\xi)=\int_{\rd} f(x+\tau y)\overline{g(x-(1-\tau)y)}e^{-2\pi iy\cdot\xi}dy.
	\]
	In particular, if $\tau=0$ and $\tau=1$ we recover the Rihaczek and the conjugate-Rihaczek distributions,  while for $\tau=1/2$ we get the classical Wigner distribution ($W=W_{1/2}$) in \eqref{CWD}.
\end{example}

\begin{example}\label{es2}
	The metaplectic operator associated to the symplectic matrix 
	\[
		\textbf{A}_{\textbf{ST}}=\begin{pmatrix}
			I_{d\times d} & -I_{d\times d} & 0_{d\times d} & 0_{d\times d}\\
			0_{d\times d} & 0_{d\times d} & I_{d\times d} & I_{d\times d}\\
			0_{d\times d} & 0_{d\times d} & 0_{d\times d} & -I_{d\times d}\\
			-I_{d\times d} & 0_{d\times d} & 0_{d\times d} & 0_{d\times d}
		\end{pmatrix}
	\]
	defines the STFT, namely $W_{\textbf{A}_{\textbf{ST}}}(f,g)=V_gf$ in \eqref{STFT}.
\end{example}

The basic continuity properties of $W_\cA$ can be summarized as follows:
\begin{prop}\label{propWellDefOpA}
	If $\mathcal{A}\in Sp(2d,\mathbb{R})$, \\
	(i) The mapping $W_\mathcal{A}:L^2(\mathbb{R}^d)\times L^2(\mathbb{R}^d)\to L^2(\mathbb{R}^{2d})$ is continuous;\\
	(ii) The mapping $W_\mathcal{A}:\mathcal{S}(\mathbb{R}^d)\times \mathcal{S}(\mathbb{R}^d)\to \mathcal{S}(\mathbb{R}^{2d})$ is continuous;\\
	(iii) The mapping $W_\mathcal{A}:\mathcal{S}'(\mathbb{R}^d)\times \mathcal{S}'(\mathbb{R}^d)\to \mathcal{S}'(\mathbb{R}^{2d})$ is continuous.
\end{prop}

%

If $\cA\in Sp(2d,\mathbb{R})$ is a general symplectic matrix, we can write explicitly $W_\cA(f,g)$ as a FIO of type II, using Lemma \ref{lemmaFreeSymp}. 

\begin{proposition}\label{propGenCase}
	Let $\cA\in Sp(2d,\mathbb{R})$ have factorization $\cA=\cA_1\cA_2$ with $\cA_j$,  $j=1,2$, free with block decomposition
	\[		\cA_j=\begin{pmatrix}
		A_j & B_j\\
		C_j & D_j
		\end{pmatrix}.	\]
 Then, up to a unitary factor, for every $F\in\mathcal{S}(\rdd)$, $(x,\xi)\in\rdd$,
 	\begin{equation}\begin{split}\label{FIO2}
	\mu(\cA)F(x,\xi)&=\left|\frac{\det(B_1)}{\det(B_2)}\right|^{1/2}\Phi_{-D_1B_1^{-1}}(x,\xi)\\
		&\qquad\times\,\int_{\mathbb{R}^{4d}}F(z,\zeta)e^{-2\pi i[\Phi_\cA(z,\zeta,y,\eta)-(x,\xi)\cdot(y,\eta)]}\tau_\cA(z,\zeta,y,\eta)dzd\zeta dyd\eta,
	\end{split}\end{equation}
	where the phase is given by
	\[
		\Phi_\cA(z,\zeta,y,\eta)=\frac{1}{2}[B_2^{-1}A_2(z,\zeta)\cdot(z,\zeta)+B_1(B_1^{-1}A_1+D_2B_2^{-1})B_1^T(y,\eta)\cdot(y,\eta)] 
		\]
	and the symbol is given by 
	\[ 
	\tau_\cA(z,\zeta,y,\eta)=e^{2\pi i B_2^{-1}B_1^T(y,\eta)\cdot(z,\zeta)},\quad z,\zeta,y,\eta\in\rd.
	\]
	In particular, for every $f,g\in\mathcal{S}(\rd)$,
	\begin{equation}\begin{split}\label{WAgeneral}
		W_\cA(f,g)(x,\xi)&=\left|\frac{\det(B_1)}{\det(B_2)}\right|^{1/2}\Phi_{-D_1B_1^{-1}}(x,\xi)\\
		&\qquad\times\,\int_{\mathbb{R}^{4d}}f(z)\overline{g(\zeta)}e^{-2\pi i[\Phi_\cA(z,\zeta,y,\eta)-(x,\xi)\cdot(y,\eta)]}\tau_\cA(z,\zeta,y,\eta)dzd\zeta dyd\eta.
		\end{split}
	\end{equation}
\end{proposition}
\begin{proof}
	Using the decomposition of Lemma \ref{lemmaFreeSymp} $(i)$, we can write, for every $F\in \mathcal{S}(\rdd)$, and up to a unitary factor,
	$$	\mu(\cA)F(x,\xi)=\mu(\cA_1)\mu(\cA_2)F(x,\xi).	$$
	For the rest of the proof, we write
	\[
		X=(x,\xi), \quad Y=(y,\eta), \quad Z=(z,\zeta),
	\]
	while $dY=dyd\eta$ and $dZ=dzd\zeta$.\\
	Applying Lemma \ref{lemmaFreeSymp} $(ii)$ twice and changing variables, up to a unitary constant,
	\begin{equation}\label{calcProp331}\begin{split}
		&\mu(\cA)F(X)=|\det(B_1)|^{-1/2}e^{-i\pi D_1B_1^{-1}X\cdot X}\int_{\rdd}(\mu(\cA_2)F)(Z)e^{2\pi i[B_1^{-1}X\cdot Y-\frac{1}{2}B_1^{-1}A_1Y\cdot Y]}dY\\
		&\qquad=|\det(B_1B_2)|^{-1/2}\!e^{-i\pi D_1B_1^{-1}X\cdot X}\!\!\int_{\rdd}\!e^{-i\pi D_2B_2^{-1}Y\cdot Y}\!\!\int_{\rdd}\!F(Z)e^{2\pi i[B_2^{-1}Y\cdot Z-\frac{1}{2}B_2^{-1}A_2Z\cdot Z]}dZ\\
		&\qquad\qquad\qquad\quad\times e^{2\pi i[B_1^{-1}X\cdot Y-\frac{1}{2}B_1^{-1}A_1Y\cdot Y]}dY\\
		&\qquad=\left|\frac{\det(B_1)}{\det(B_2)}\right|^{1/2}\Phi_{-D_1B_1^{-1}}(X)\!\!\int_{\mathbb{R}^{4d}}F(Z)e^{2\pi i[-\frac{1}{2}(B_1(B_1^{-1}A_1+D_2B_2^{-1})B_1^T Y\cdot Y+B_2^{-1}A_2Z\cdot Z)+X\cdot Y]}\\
		&\qquad\qquad\qquad\quad\times e^{2\pi i B_2^{-1}B_1^TY\cdot Z}dYdZ.
	\end{split}\end{equation}
	This proves (\ref{FIO2}) and (\ref{WAgeneral}) follows plugging $F(z,\zeta)=(f\otimes\bar{g})(z,\zeta)$ in (\ref{FIO2}).
\end{proof}

Moreover, we also have explicit integral formulas for metaplectic Wigner distributions in terms of their factorization via free symplectic matrices.

\begin{corollary}\label{corAlTeoPrecedente}
	Under the same notation of Proposition \ref{propGenCase}, up to a unitary factor and for every $f,g\in\mathcal{S}(\rd)$,
	\begin{equation}\label{WAgeneral}\begin{split}
		&W_\cA(f,g)(x,\xi)=|\det(B_1B_2)|^{-1/2}\Phi_{-D_1B_1^{-1}}(x,\xi)\\
		&\qquad\times\int_{\rdd}\!\!f(y)\overline{g(\eta)}\Phi_{-B_2^{-1}A_2}(y,\eta)\mathcal{F}^{-1}(\Phi_{-(B_1^{-1}A_1+D_2B_2^{-1})})(B_1^{-1}(x,\xi)+B_2^{-T}(y,\eta))dyd\eta\\
		&\quad\quad=|\det(B_1B_2)|^{-1/2}\Phi_{-D_1B_1^{-1}}(x,\xi)\\
		&\quad\quad\times[((f\otimes\bar{g})\Phi_{-B_2^{-1}A_2})\ast(\mathcal{F}^{-1}\Phi_{-(B_1^{-1}A_1+D_2B_2^{-1})}\circ(-B_2^{-T}))]\circ(-B_2^TB_1^{-1})(x,\xi),\notag
	\end{split}
	\end{equation}
where the chirp $\Phi$ is defined in \eqref{chirp}.  In particular, if $B_1^{-1}A_1+D_2B_2^{-1}$ is invertible, then, up to a phase factor,
	\begin{equation}\label{WAgeneralMod}
	\begin{split}
	&W_\cA(f,g)(x,\xi)=|\det(B_1B_2)|^{-1/2}|\det(B_1^{-1}A_1+D_2B_2^{-1})|^{-1}\Phi_{-D_1B_1^{-1}}(x,\xi)\notag\\
		&\qquad\times\int_{\rdd}f(y)\overline{g(\eta)}\Phi_{-B_2^{-1}A_2}(y,\eta)\Phi_{(B_1^{-1}A_1+D_2B_2^{-1})^{-1}}(B_1^{-1}(x,\xi)+B_2^{-T}(y,\eta))dyd\eta.
	\end{split}
	\end{equation}
\end{corollary}
\begin{proof}
	Using (\ref{FIO2}), we can write, for every $F\in \mathcal{S}(\rdd)$, and up to a unitary factor,
	$$	\mu(\cA)F(x,\xi)=\mu(\cA_1)\mu(\cA_2)F(x,\xi).	$$
	Applying Lemma \ref{lemmaFreeSymp} $(ii)$ twice, up to a unitary constant,
	\begin{equation}\label{calcProp331}\begin{split}
		&\mu(\cA)F(x,\xi)=|\det(B_1B_2)|^{-1/2}e^{-i\pi D_1B_1^{-1}(x,\xi)\cdot (x,\xi)}\int_{\rdd}F(y,\eta)e^{-\pi iB_2^{-1}A_2(y,\eta)\cdot (y,\eta)}\\
		&\qquad\quad\times \int_{\rdd} e^{2\pi i[B_1^{-1}(x,\xi)\cdot (z,\zeta)-\frac{1}{2}B_1^{-1}A_1(z,\zeta)\cdot (z,\zeta)-\frac{1}{2}D_2B_2^{-1}(z,\zeta)\cdot (z,\zeta)+B_2^{-1}(z,\zeta)\cdot (y,\eta)]}dzd\zeta dyd\eta\\
		&\qquad=|\det(B_1B_2)|^{-1/2}\Phi_{-D_1B_1^{-1}}(x,\xi)\int_{\rdd}F(y,\eta)\Phi_{-B_2^{-1}A_2}(y,\eta)\\
		&\qquad\quad\times \int_{\rdd} e^{2\pi i[(B_1^{-1}(x,\xi)+B_2^{-T}(y,\eta))\cdot(z,\zeta)-\frac{1}{2}(B_1^{-1}A_1+D_2B_2^{-1})(z,\zeta)\cdot(z,\zeta)]}dzd\zeta.
	\end{split}\end{equation}
The inner integral is worked out as
	\begin{equation}\label{calcProp332}\begin{split}
		\int_{\rdd}& e^{2\pi i[(B_1^{-1}(x,\xi)+B_2^{-T}(y,\eta))\cdot(z,\zeta)-\frac{1}{2}(B_1^{-1}A_1-D_2B_2^{-1})(z,\zeta)\cdot(z,\zeta)]}dzd\zeta\\
		&=\int_{\rdd}\Phi_{-(B_1^{-1}A_1-D_2B_2^{-1})}(z,\zeta)e^{2\pi i[B_1^{-1}(x,\xi)+B_2^{-T}(y,\eta)]\cdot(z,\zeta)}dzd\zeta\\
		&=\mathcal{F}^{-1}(\Phi_{-(B_1^{-1}A_1+D_2B_2^{-1})})(B_1^{-1}(x,\xi)+B_2^{-T}(y,\eta)).
	\end{split}\end{equation}
	Observe that if $B_1^{-1}A_1+D_2B_2^{-1}$ is invertible, then 
	\[
	\mathcal{F}^{-1}(\Phi_{-(B_1^{-1}A_1+D_2B_2^{-1})})=\frac{1}{|\det(B_1^{-1}A_1+D_2B_2^{-1})|}\Phi_{(B_1^{-1}A_1+D_2B_2^{-1})^{-1}}.
	\]
	Pluggin (\ref{calcProp332}) into (\ref{calcProp331}) with $F(y,\eta)=f(y)\overline{g(\eta)}$ the assertion follows.
\end{proof}

The integral expression of $W_\cA$ provided by Corollary \ref{corAlTeoPrecedente} is useful to establish continuity properties for $W_\cA(f,g)$. In practice, the factorization of $\cA$ via free matrices may be unknown.

	{\begin{definition}\label{defMetaplPsiDo}
	Let $a\in\mathcal{S}'(\mathbb{R}^{2d})$. The \textbf{metaplectic pseudodifferential operator} with \textbf{symbol} $a$ and symplectic matrix $\cA$ is the operator $Op_\mathcal{A}(a):\mathcal{S}(\mathbb{R}^d)\to\mathcal{S}'(\mathbb{R}^{d})$ such that
	\[		\langle Op_\mathcal{A}(a)f,g\rangle=\langle a,W_\mathcal{A}({g},f)\rangle, \quad g\in\mathcal{S}(\mathbb{R}^d).	\]
	\end{definition}
	Observe that this operator is well defined by Proposition \ref{propWellDefOpA}, item $(iii)$. Moreover, when the context requires to stress the matrix $\cA$ that defines $Op_\cA$, we refer to $Op_\cA$ to as the \textbf{$\cA$-pseudodifferential operator} with {symbol} $a$. }
	{\begin{remark}
		In principle, the full generality of metaplectic framework provides a wide variety of unexplored time-frequency representations that fit many different contexts. Namely, in Definition \ref{defMetaplPsiDo}, the symplectic matrix $\cA$ plays the role of a quantization and the quantization of a pseudodifferential operator is typically chosen depending on the the properties that must be satisfied in a given setting.
	\end{remark}}
	
	\begin{example}
		Definition \ref{defMetaplPsiDo} in the case of $\cA_{1/2,2d}\in Sp(2d,\mathbb{R})$ in Example \ref{es1}, provides the well-known Weyl quantization for pseudodifferential operators, that we denote with $Op_{w,2d}(a)$, i.e.
		\[
			Op_{w,2d}(a)f(x)=\int_{\rdd}a\left(\frac{x+y}{2},\xi\right)f(y)e^{2\pi i(x-y)\cdot\xi}dyd\xi, \quad f\in\mathcal{S}(\rd).
		\]
		When $d$ is clear from the context or irrelevant, we write $Op_w$ instead of $Op_{w,2d}$.
	\end{example}
	
	In the following result, we see how the symbols of {metaplectic pseudodifferential operators} change when we modify the symplectic matrix. 
	
	\begin{lemma}\label{lemmaComm}
		Consider $\mathcal{A},\mathcal{B}\in Sp(2d,\mathbb{R})$ and $a,b\in\mathcal{S}'(\mathbb{R}^{2d})$. Then, 
		\begin{equation}\label{changeMatrix}Op_{\mathcal{A}}(a)=Op_\mathcal{B}(b)\quad \Longleftrightarrow\,\,\quad
		b=\mu(\mathcal{B}\mathcal{A}^{-1})(a).
		\end{equation}
	\end{lemma}
	\begin{proof}
		Let $f,g\in\mathcal{S}(\mathbb{R}^d)$. Then,
		\begin{align*}
			&\langle Op_\mathcal{A}(a)f,g\rangle=\langle a,\mu(\mathcal{A})(f\otimes \bar{g})\rangle=\langle \mu(\mathcal{A}^{-1})a,f\otimes \bar{g}\rangle,\\
			&\langle Op_\mathcal{B}(b)f,g\rangle=\langle b,\mu(\mathcal{B})(f\otimes \bar{g})\rangle=\langle \mu(\mathcal{B}^{-1})b,f\otimes \bar{g}\rangle.
		\end{align*}
		Since $\mathcal{S}(\mathbb{R}^d)\otimes\mathcal{S}(\mathbb{R}^d)$ is dense in $\mathcal{S}(\mathbb{R}^{2d})$, we deduce that the equality between the two lines holds if and only if
		\[
			\mu(\mathcal{A}^{-1})a=\mu(\mathcal{B}^{-1})b,
		\]
		which is the same as (\ref{changeMatrix}).
	\end{proof}

As a direct consequence of Lemma \ref{lemmaComm} we get two corollaries. The first one provides the distributional kernel of $Op_\cA$.
	
	\begin{cor} Consider $\cA\in Sp(2d,\mathbb{R})$, $a\in \mathcal{S}'(\rdd)$. Then, for all $f,g\in\mathcal{S}(\rd)$,
	\begin{equation}\label{kernelMPO}
		\langle Op_\cA(a)f,g\rangle=\langle k_\cA(a),g\otimes \bar f\rangle,		
	\end{equation}
	where the kernel is given by $k_\cA(a)=\mu(\cA^{-1})a$. 
	\end{cor}
	\begin{proof}
	Plug $\mathcal{B}=I_{4d\times 4d}$ into (\ref{changeMatrix}) to get (\ref{kernelMPO}). 
	\end{proof}
	
	Corollary \ref{corollarioallacomm} is a generalization of \eqref{I3} for metaplectic Wigner distributions and pseudodifferential operators. For its statement, we introduce the following notation: if $a\in\mathcal{S}'(\rdd)$, $a\otimes 1$ denotes the tempered distribution of $\mathcal{S}'(\mathbb{R}^{4d})$ defined via tensor product as 
	\begin{equation}\label{defsigma}
	(a\otimes 1)(r,y,\rho,\eta):=a(r,\rho),\quad r,y,\rho,\eta\in\rd.
	\end{equation} 
	
	\begin{cor}\label{corollarioallacomm}
		Consider $\cA\in Sp(4d,\mathbb{R})$, $\mathcal{B}\in Sp(2d,\mathbb{R})$ and $a\in\mathcal{S}'(\mathbb{R}^{2d})$. Then, for all $\mathcal{B}_0\in Sp(4d,\mathbb{R})$, $f,g\in\mathcal{S}(\rd)$,
		\begin{equation}\label{commOpW}
			W_\cA(Op_{\mathcal{B}}(a)f,{g})=Op_{\mathcal{B}_0}(\mu(\mathcal{B}_0\mathcal{A}_{1/2,4d}^{-1})((\mu(\mathcal{A}_{1/2,2d}\mathcal{B}^{-1})a)\otimes 1)\circ\cA^{-1})W_\cA(f,g).
		\end{equation}
		In particular, \\
		(i) if $\mathcal{B}_0=\cA_{1/2,4d}$, then
		\begin{equation}\label{commOpWParticolare}
				W_\cA(Op_\mathcal{B}(a)f,g)=Op_{w,4d}(((\mu(\mathcal{A}_{1/2,2d}\mathcal{B}^{-1})a)\otimes 1)\circ\cA^{-1})W_\cA(f,g);
		\end{equation}
		\noindent
		(ii) if $\mathcal{B}_0=\cA_{1/2,4d}$ and $\mathcal{B}=\cA_{1/2,2d}$, then
		\begin{equation}\label{commOpWParticolare2}
			W_\cA(Op_{w,2d}(a)f,g)=Op_{w,4d}((a\otimes 1)\circ\cA^{-1})W_\cA(f,g).
		\end{equation}
	\end{cor}
	\begin{proof}
		By \cite[Lemma 4.1]{CR2021}, for all $f,g\in\mathcal{S}(\rd)$ and $a\in\mathcal{S}'(\rdd)$,
		\begin{equation}\label{CR202141}
			(Op_{w,2d}(a)f)\otimes\bar g=Op_{w,4d}(\sigma)(f\otimes \bar g).
		\end{equation}
		Moreover, for all $\mathcal{A}\in Sp(4d,\mathbb{R})$, 
		\begin{equation}
		\label{commclassW}
			\mu(\cA)Op_{w,4d}(\sigma)\mu(\cA)^{-1}=Op_{w,4d}(\sigma\circ\cA^{-1}).
		\end{equation}
		Therefore, using Lemma \ref{lemmaComm}, (\ref{CR202141}) and (\ref{commclassW}) respectively,
		\begin{align*}
			W_\cA(Op_\mathcal{B}(a)f,g)&=\mu(\cA)(Op_{\mathcal{B}}(a)f\otimes\bar g)=\mu(\cA)(Op_{w,2d}(\mu(\mathcal{A}_{1/2,2d}\mathcal{B}^{-1})a)f\otimes \bar g)\\
			&=\mu(\cA)(Op_{w,4d}((\mu(\mathcal{A}_{1/2,2d}\mathcal{B}^{-1})a)\otimes 1)(f\otimes\bar g))\\
			&=Op_{w,4d}(((\mu(\mathcal{A}_{1/2,2d}\mathcal{B}^{-1})a)\otimes 1)\circ\cA^{-1})\mu(\cA)(f\otimes\bar g)\\
			&=Op_{w,4d}(((\mu(\mathcal{A}_{1/2,2d}\mathcal{B}^{-1})a)\otimes 1)\circ\cA^{-1})W_\cA(f,g).
		\end{align*}
		Then, by Lemma \ref{lemmaComm},
		\begin{align*}
			W_\cA(Op_\mathcal{B}(a)f,g)&=Op_{\mathcal{B}_0}(\mu(\mathcal{B}_0\mathcal{A}_{1/2,4d}^{-1})((\mu(\mathcal{A}_{1/2,2d}\mathcal{B}^{-1})a)\otimes 1)\circ\cA^{-1})W_\cA(f,g)
		\end{align*}
		and we are done.
	\end{proof}
	
	\begin{remark}
	Formula (\ref{commOpW}) will be used in the form of (\ref{commOpWParticolare2}) to deduce boundedness properties on modulation spaces for metaplectic pseudodifferential operators. However, the strength of Corollary \ref{corollarioallacomm} relies on its generality: the matrix $\mathcal{B}_0$ in (\ref{commOpW}) can be chosen in $Sp(4d,\mathbb{R})$ arbitrarily, depending on the context. 
	\end{remark}

\section{Decomposability and covariance}

In this section, we focus on metaplectic Wigner distribution as well as metaplectic pseudodifferential operators that are defined in terms of symplectic matrices that satisfy {decomposability} and covariance properties. We will prove that if a matrix $\cA$ satisfies one of these properties, it can be easily factorized to get an explicit expressions for $W_\cA$ and $Op_\cA$ in terms of its blocks.


\subsection{Decomposability and shift-invertibility}
We define decomposable metaplectic Wigner distributions directly in terms of their factorization, as follows. Let $\mathcal{A}$ be a symplectic matrix that factorizes as 
\begin{equation}\label{SympDec}
\mathcal{A}=\mathcal{A}_{FT2}\mathcal{D}_L,
\end{equation} 
where $\mathcal{D}_L$ is defined in \eqref{F2} and 
\[
	\mathcal{A}_{FT2}=\begin{pmatrix}
		I_{d\times d} & 0_{d\times d} & 0_{d\times d} & 0_{d\times d}\\
		0_{d\times d} & 0_{d\times d} & 0_{d\times d} & I_{d\times d}\\
		0_{d\times d} & 0_{d\times d} & I_{d\times d} & 0_{d\times d}\\
		0_{d\times d} & -I_{d\times d} & 0_{d\times d} & 0_{d\times d}
	\end{pmatrix}.
\]
Up to a phase factor, 
\[
	\mu(\mathcal{A}_{FT2})=\cF_2.
\]

\begin{definition}\label{WigDecDef} We say that $\mathcal{A}\in Sp(2d,\mathbb{R})$ is a \textbf{totally Wigner-decomposable} (symplectic) matrix if (\ref{SympDec}) holds for some $L\in GL(2d,\mathbb{R})$. If $\cA$ is totally Wigner-decomposable, we say that $W_\cA$ is of the \textbf{classic type}.
\end{definition} 

\begin{example}
	The matrices of Examples \ref{es1} and \ref{es2} are totally Wigner-decomposable with
	\[
		L_\tau=\begin{pmatrix}
			I_{d\times d} & \tau I_{d\times d}\\
			I_{d\times d} & -(1-\tau) I_{d\times d}
		\end{pmatrix}
	\]
	and
	\[
		L_{ST}=\begin{pmatrix} 0_{d\times d} & I_{d\times d}\\ -I_{d\times d} & I_{d\times d} \end{pmatrix}=J
	\]
	($J$ defined in \eqref{J}) respectively.
\end{example}


Roughly speaking, Wigner distributions of classic type are immediate generalizations of the classical time-frequency representations, such as the (cross)-Wigner distribution $W$ and the STFT.\\

The following result characterizes totally Wigner-decomposable symplectic matrices in terms of their block decomposition.

\begin{proposition}\label{explA}
	Let $\mathcal{A}\in Sp(2d,\mathbb{R})$ be a totally Wigner-decomposable matrix 
	having block decomposition
		\begin{equation}\label{blockAgen}
	\mathcal{A}=\begin{pmatrix}
	A_{11} & A_{12} & A_{13} & A_{14}\\
	A_{21} & A_{22} & A_{23} & A_{24}\\
	A_{31} & A_{32} & A_{33} & A_{34}\\
	A_{41} & A_{42} & A_{43} & A_{44}
	\end{pmatrix},
\end{equation}
	with $A_{ij}\in\mathbb{R}^{d\times d}$ ($i,j=1,\ldots,4$). Then, \\
(i) $\mathcal{A}$ has the block decomposition 
\begin{equation}\label{Arepr}
\mathcal{A}=\begin{pmatrix}
	A_{11} & A_{12} & 0_{d\times d} & 0_{d\times d}\\
	0_{d\times d} & 0_{d\times d} & A_{23} & A_{24}\\
	0_{d\times d} & 0_{d\times d} & A_{33} & A_{34}\\
	A_{41} & A_{42} & 0_{d\times d} & 0_{d\times d}
	\end{pmatrix};
\end{equation} 
	(ii) $L$ and its inverse are related to $\mathcal{A}$ by:
	\begin{equation}\label{relAL}
	L=\begin{pmatrix}
		A_{33}^T & A_{23}^T\\
		A_{34}^T & A_{24}^T
	\end{pmatrix}, \ \ \ L^{-1}=\begin{pmatrix}
			A_{11} & A_{12}\\
			-A_{41} & -A_{42}
		\end{pmatrix}.
	\end{equation}
\end{proposition}
\begin{proof}
Let
\begin{equation}\label{decompL}
	L=\begin{pmatrix}
	L_{11} & L_{12}\\
	L_{21} & L_{22}
	\end{pmatrix} \ \ \text{and} \ \ L^{-1}=\begin{pmatrix}
	L_{11}' & L_{12}'\\
	L_{21}' & L_{22}'
	\end{pmatrix}
\end{equation}
be the block decompositions of $L$ and $L^{-1}$ respectively, where $L_{ij}, L_{ij}'\in \mathbb{R}^{d\times d}$ ($i,j=1,2$). Then, the identity (\ref{SympDec}) reads as
\[
	\begin{pmatrix}
	A_{11} & A_{12} & A_{13} & A_{14}\\
	A_{21} & A_{22} & A_{23} & A_{24}\\
	A_{31} & A_{32} & A_{33} & A_{34}\\
	A_{41} & A_{42} & A_{43} & A_{44}
	\end{pmatrix}=\begin{pmatrix}
		L_{11}' & L_{12}' & 0_{d\times d} & 0_{d\times d}\\
		0_{d\times d} & 0_{d\times d} & L_{12}^T & L_{22}^T\\
		0_{d\times d} & 0_{d\times d} & L_{11}^T & L_{21}^T\\
		-L_{21}' & -L_{22}' & 0_{d\times d} & 0_{d\times d}
	\end{pmatrix}.
\]
Thus the expressions for the matrices in $(i)$ and $(ii)$ easily follow. 
\end{proof}

\begin{remark}
	Under the hypothesis of Proposition \ref{explA}, it is easy to check that the identities $LL^{-1}=L^{-1}L=I_{2d\times2d}$ read in terms of the blocks of $L$ and $L^{-1}$ as
	\begin{equation}\label{sympRelAAm1}
		\begin{cases}
			A_{33}^TA_{11}-A_{23}^TA_{41}=I_{d\times d},\\
			A_{33}^TA_{12}=A_{23}^TA_{42},\\
			A_{34}^TA_{11}=A_{24}^TA_{41},\\
			A_{34}^TA_{12}-A_{24}^TA_{42}=I_{d\times d}
		\end{cases} \ \ \text{and} \ \ \begin{cases}
		A_{11}A_{33}^T+A_{12}A_{34}^T=I_{d\times d},\\
			A_{11}A_{23}^T=-A_{12}A_{24}^T,\\
			A_{41}A_{33}^T=-A_{42}A_{34}^T,\\
			A_{41}A_{23}^T+A_{42}A_{24}^T=-I_{d\times d}.
		\end{cases}
	\end{equation}
	These are exactly the block relations that $\mathcal{A}$ and
	 \begin{equation}\label{Ainv}
	 	\mathcal{A}^{-1}=\mathcal{D}_L^{-1}\mathcal{A}_{FT2}^{-1}=\begin{pmatrix}
			A_{33}^T & 0_{d\times d} & 0_{d\times d} & -A_{23}^T\\
			A_{34}^T & 0_{d\times d} & 0_{d\times d} & -A_{24}^T\\
			0_{d\times d} & -A_{41}^T & A_{11}^T & 0_{d\times d}\\
			0_{d\times d} & -A_{42}^T & A_{12}^T & 0_{d\times d}
		\end{pmatrix}
	\end{equation}
	satisfy as symplectic matrices.
\end{remark}

\begin{definition}\label{shift-invertible}
	Given $\cA\in Sp(2d,\bR)$, we say that $W_{\cA}$ (by abuse, $\cA$) is \emph{\bf shift-invertible} if there exists an invertible matrix $E_\cA\in GL(2d,\bR)$ such that
	$$ |W_{\cA}(\pi(w)f,g)|=|T_{E_\cA(w)}W_{\cA}(f,g)|,\quad f,g\in\lrd,\quad w\in\rdd,$$
	where $$T_{E_\cA(w)}W_{\cA}(f,g)(z)=W_{\cA}(f,g)(z-E_\cA w),\quad w,z\in\rdd.$$
\end{definition}

As pointed out in \cite{CR2022}, shift-invertibility of symplectic matrices appears to be the fundamental property that a metaplectic Wigner distribution shall satisfy in order for $W_\cA(\cdot,g)$ to replace the STFT in the definition of modulation spaces.

\begin{lemma}\label{ShiftInvert}
	Let $\mathcal{A}\in Sp(2d,\mathbb{R})$ be a totally Wigner-decomposable as in (\ref{SympDec}). The following statements are equivalent:\\
	(i) $L$ is right-regular\footnote{\begin{definition}
			A $d$-block matrix $L=\left(\begin{array}{cc}
				L_{11} & L_{12}\\
				L_{21} & L_{22}
			\end{array}\right)\in\mathbb{R}^{2d\times2d}$ is called left-regular (resp. right-regular) if the submatrices $L_{11},L_{21}\in\mathbb{R}^{d\times d}$
			(resp. $L_{12},L_{22}\in\mathbb{R}^{d\times d}$) are invertible. 
	\end{definition}};\\
	(ii)
	the matrix
	\begin{equation}\label{defEA}
	E_{\mathcal{A}}:=\begin{pmatrix}
	A_{11} & 0_{d\times d}\\
	0_{d\times d} & A_{23}
	\end{pmatrix}
	\end{equation}
	is invertible;\\
	(iii) $W_{\cA}$ is shift-invertible with $E_\mathcal{A}$ given as in (\ref{defEA}).
\end{lemma}
\begin{proof}
	The equivalence between $(ii)$ and $(iii)$ is proved in \cite{CR2022}. We prove that $(i)$ and $(ii)$ are equivalent.
	
	$(i)\Rightarrow (ii)$. Assume that $L$ is right-regular. We have to prove that both $A_{23}$ and $A_{11}$ are invertible. The right-regularity of $L$ is equivalent to the invertibility of $A_{23}$ and $A_{24}$, hence it remains to check that $A_{11}$ is invertible. 

	It is easy to verify that $L$ is right-regular if and only if $L^{-T}$ is left-regular. By Proposition \ref{explA} $(ii)$,
	\[
		L^{-T}=\begin{pmatrix}
			A_{11}^T & -A_{41}^T\\
			A_{12}^T & -A_{42}^T
		\end{pmatrix},
	\]
	so that $L^{-T}$ is left-regular if and only if $A_{11}$ and $A_{12}$ are invertible, which gives the invertibility of $A_{11}$.
	
	$(i)\Leftarrow (ii)$. If $E_\mathcal{A}$ is invertible, then $A_{11}$ and $A_{23}$ are invertible. By the identity $A_{11}A_{23}^T=-A_{12}A_{24}^T$ in (\ref{sympRelAAm1}), we also have the invertibility of $A_{12}$ and $A_{24}$. Hence, $A_{23}$ and $A_{24}$ are invertible.
\end{proof}
%

\begin{corollary}\label{corReprL}
	Let $\mathcal{A}$ satisfy (\ref{SympDec}) with block decomposition as in (\ref{Arepr}). The following statements are equivalent:\\
	(i) $L$ is right-regular;\\
	(ii) $A_{11}$, $A_{12}$, $A_{23}$ and $A_{24}$ are invertible.\\
	Moreover, if $L$ is right-regular, \\
	(iii) $A_{33}$ is invertible if and only if $A_{42}$ is invertible;\\
	(iv) $A_{34}$ is invertible if and only if $A_{41}$ is invertible.
\end{corollary}
\begin{proof}
	The equivalence between $(i)$ and $(ii)$ is just a restatement of Lemma \ref{ShiftInvert}. $(iii)$ and $(iv)$ follow directly from $(ii)$ and the equalities in \eqref{sympRelAAm1}.
\end{proof}

\begin{remark}
	Assume that $L$ is right-regular with block decomposition as in (\ref{decompL}). Since $L$ is also invertible by its definition, all the assumptions of Theorem 2.1 (ii) and Theorem 2.2 (i) of \cite{LS} are verified. Thus, we can write a Wigner-decomposable matrix $\mathcal{A}$, with $L$ right-regular, explicitly in terms of the blocks of $L$ both as
	\[
		\mathcal{A}=\begin{pmatrix}
			A_{11} & A_{12} & 0_{d\times d} & 0_{d\times d}\\
			0_{d\times d} & 0_{d\times d} & L_{12}^T & L_{22}^T\\
			0_{d\times d} & 0_{d\times d} & L_{11}^T & L_{21}^T\\
		A_{41} & A_{42} & 0_{d\times d} & 0_{d\times d}
		\end{pmatrix}
	\]
	with
	\begin{align*}
	A_{11}&=(L_{11}-L_{12}L_{22}^{-1}L_{21})^{-1},\\
	A_{12}&=-(L_{11}-L_{12}L_{22}^{-1}L_{21})^{-1}L_{12}L_{22}^{-1},\\
	A_{41}&=L_{22}^{-1}L_{21}(L_{11}-L_{12}L_{22}^{-1}L_{21})^{-1},\\ A_{42}&=-L_{22}^{-1}-L_{22}^{-1}L_{21}(L_{11}-L_{12}L_{22}^{-1}L_{21})^{-1}L_{12}L_{22}^{-1};
	\end{align*}
	 or, equivalently, as
	\[\mathcal{A}=\begin{pmatrix}
		A_{11} & A_{12}& 0_{d\times d} & 0_{d\times d}\\
		0_{d\times d} & 0_{d\times d} & L_{12}^T & L_{22}^T\\
		0_{d\times d} & 0_{d\times d} & L_{11}^T & L_{21}^T\\
		 A_{41}& A_{42} & 0_{d\times d} & 0_{d\times d}\end{pmatrix},
	\]
where 
\begin{align*}
A_{11}&=-(L_{21}-L_{22}L_{12}^{-1}L_{11})^{-1}L_{22}L_{12}^{-1},\\ A_{12}&=(L_{21}-L_{22}L_{12}^{-1}L_{11})^{-1},\\ A_{41}&=-L_{12}^{-1}-L_{12}^{-1}L_{11}(L_{21}-L_{22}L_{12}^{-1}L_{11})^{-1}L_{22}L_{12}^{-1}\\ A_{42}&=-L_{12}^{-1}L_{11}(L_{21}-L_{22}L_{12}^{-1}L_{11})^{-1}L_{22}L_{12}^{-1}.
\end{align*}
\end{remark}

We now prove a generalization of \cite[Theorem 2.28]{CR2022}, see also \cite{BCGT20}. 

\begin{theorem}\label{thm227revised}
	Let $L$ be right-regular and $\mathcal{A}$ be as in (\ref{Arepr}). Then, up to a phase factor, for all $f,g\in L^2(\mathbb{R}^d)$ and  $x,\xi\in\mathbb{R}^d$,
	\begin{equation}\label{WaVg}
		W_\mathcal{A}(f,g)(x,\xi)=\sqrt{|\det(L)|}|\det(A_{23})|^{-1}e^{2\pi iA_{23}^{-1}\xi\cdot A_{33}^Tx}V_{\widetilde{g}}f(c(x),d(\xi)),
	\end{equation}
	where
	\begin{equation}\label{defgcd}
		\widetilde{g}(t):=g(A_{24}^T(A_{23}^T)^{-1}t), \ c(x)=(A_{33}^T-A_{23}^T(A_{24}^T)^{-1}A_{34}^T)x \ \text{and} \ d(\xi)=A_{23}^{-1}\xi.
	\end{equation}
\end{theorem}
Observe that all the inverses that appear in (\ref{WaVg}) exist if $L$ is right-regular by Corollary \ref{corReprL} $(ii)$ and Theorem 2.1 $(ii)$ of \cite{LS}.
\begin{proof}
	The proof is a straightforward consequence of \cite[Theorem 3.8]{CT2020}. 
\end{proof}

\begin{theorem}\label{thm28revised}
	Let $0<p,q\leq\infty$, $L$ be right-regular and $\mathcal{A}$ be as in (\ref{Arepr}). Let $m\in\mathcal{M}_v$ be such that 
	\begin{equation}\label{assumptOnm}
		m((A_{33}^T-A_{23}^T(A_{24}^T)^{-1}A_{34}^T)\cdot,A_{23}^{-1}\cdot)\asymp {m(\cdot,\cdot)}.
	\end{equation}
	Then, for all $g\in\mathcal{S}(\mathbb{R}^d)$,
	\begin{equation}\label{charMods}
		f\in M^{p,q}_m(\rd) \ \ \ \Longleftrightarrow \ \ \ W_\mathcal{A}(f,g)\in L^{p,q}_m(\mathbb{R}^{2d}).
	\end{equation}
	Moreover, if $1\leq p,q\leq\infty$ and there exist $0<C_1(L)\leq C_2(L)$ such that
	\begin{equation}\label{reqOnv}
		C_1(L)v(x,\xi)\leq v((A_{23}^{T}(A_{24}^T)^{-1})x,A_{23}^{-1}A_{24}\xi)\leq C_2(L)v(x,\xi), \quad (x,\xi)\in\rdd,
	\end{equation}
	then $g$ can be chosen in the larger class $M^1_v(\mathbb{R}^d)$.
\end{theorem}
\begin{proof}
	The proof is a straightforward consequence of Theorem \ref{thm227revised}. In fact, for $g\in\mathcal{S}(\mathbb{R}^d)$ and $L$ right-regular, the function $\widetilde{g}$ defined as in (\ref{defgcd}) is in $\mathcal{S}(\mathbb{R}^d)$ and by (\ref{WaVg}),
\begin{align*}
		\Vert f\Vert_{M^{p,q}_m}&\asymp \Vert V_{\widetilde{g}}f\Vert_{L^{p,q}_m}\asymp \Vert W_\mathcal{A}(f,g)((A_{33}^T-A_{23}^T(A_{24}^T)^{-1}A_{34}^T)^{-1}\cdot,A_{23}\cdot)\Vert_{L^{p,q}_m}\\
		&\asymp\Vert W_{\mathcal{A}}(f,g)\Vert_{L^{p,q}_m},
\end{align*}
	by assumption (\ref{assumptOnm}).
	
	Assume that $\varphi\in \mathcal{S}(\mathbb{R}^d)$. Then, 
	\begin{align*}
		V_\varphi\widetilde{g}(x,\xi)&=\int_{\mathbb{R}^d}g(Ct)e^{-2\pi i\xi \cdot t}\overline{\varphi(t-x)}dt\asymp \int_{\mathbb{R}^d}g(s)e^{-2\pi i((C^{-1})^T\xi)\cdot s}\overline{\varphi(C^{-1}(s-Cx))}ds\\
		&=V_{\widetilde{\varphi}}g(B(x,\xi)),
	\end{align*}
	where
	\[
		C=A_{24}^TA_{23}^{-T}\text{,} \quad\quad B=\begin{pmatrix} C & 0_{d\times d}\\ 0_{d\times d} & C^{-T}
		\end{pmatrix}=\mathcal{D}_{A_{24}^{-1}}\mathcal{D}_{A_{23}^T}
	\]
	and $\widetilde\varphi(t)=\varphi(C^{-1}t)$. Condition (\ref{reqOnv}) implies that $g\in M^1_v$ if and only if $\widetilde{g}\in M^1_v$: 
\begin{align*}	
		\Vert \widetilde{g}\Vert_{M^1_{v}}&\asymp\Vert V_\varphi \widetilde{g}\Vert_{L^1_v}\asymp\Vert V_{\widetilde\varphi}g(B\cdot)v(\cdot)\Vert_{L^1}
		\asymp\Vert V_{\widetilde\varphi}g(\cdot) v(B^{-1}\cdot)\Vert_{L^1} \\
		&\asymp\Vert V_{\widetilde\varphi}g(\cdot)v(\cdot)\Vert_{L^1}
		\asymp\Vert g\Vert_{M^1_v}.
	\end{align*}
	Hence, for $1\leq p,q\leq\infty$, we can choose $g$ in $M^1_v(\mathbb{R}^d)$.
\end{proof}

Next, we generalize the metaplectic Wigner distributions associated to Wigner-decomposable matrices in order to include multiplications by chirps. These Wigner distributions, along with the right-regularity condition on $L$, characterize modulation spaces. 
\begin{definition}
	We say that a matrix $\mathcal{A}\in Sp(d,\mathbb{R})$ is \textbf{Wigner-decomposable} if $\mathcal{A}=V_C\mathcal{A}_{FT2}\mathcal{D}_L$ for some $C\in\rd$ symmetric and $L\in GL(d,\mathbb{R})$.
\end{definition}
\begin{theorem}\label{charWAtotWDec}
		Let $\mathcal{A}\in Sp(2d,\mathbb{R})$ be a Wigner-decomposable matrix with decomposition $\cA=V_C\cA_{FT2}\mathcal{D}_L$
		\[
			C=\begin{pmatrix}
				C_{11} & C_{12}\\
				C_{12}^T & C_{22}
			\end{pmatrix}
		\]
		$C_{11}^T=C_{11}$ and $C_{22}^T=C_{22}$. Then,  for all $f,g\in L^2(\mathbb{R}^d)$,  up to a phase factor,
		\begin{equation}\label{CtypeFormula}\begin{split}
			W_\cA(f,g)(x,\xi)=\widetilde{\Phi}_{C}(x,\xi)
			\int_{\mathbb{R}^d} f(x+(I_{d\times d}-A_{11})y)\overline{g(x-A_{11}y)}e^{-2\pi i\xi \cdot y}dy,
		\end{split}\end{equation}
		with 
		$$\widetilde{\Phi}_{C}(x,\xi)=e^{2\pi i C_{12}^Tx\cdot\xi}\Phi_{C_{11}}(x)\Phi_{C_{22}}(\xi)$$
		and the chirp functions $\Phi_{C_{11}}$, $\Phi_{C_{22}}$ are defined in \eqref{chirp}.
		If the matrix $L$ is right-regular, 
		\begin{equation}\label{WASTFTCtype}
		\begin{split}
			W_\mathcal{A}(f,g)(x,\xi)= |\det(I_{d\times d}-A_{11})|^{-1} \widetilde{\Phi}'_{C}(x,\xi)
			 V_{\tilde{g}}f(A_{11}^{-1}x,(I-A_{11}^T)^{-1}\xi),
		\end{split}
		\end{equation}
		where
		$$\widetilde{\Phi}'_{C}(x,\xi)=\Phi_{C_{11}}(x)\Phi_{C_{22}}(\xi)e^{2\pi i( C_{12}^T+(I-A_{11})^{-1})x\cdot\xi}, $$
		and $\tilde{g}(t)=g(-A_{11}(I_{d\times d}-A_{11})^{-1}t)$.
	\end{theorem}
	\begin{proof}
		Formula (\ref{CtypeFormula}) is proved using the explicit definitions of the operators associated to $V_C$, $\mathcal{A}_{FT2}$ and $\mathcal{D}_L$. In fact, up to a phase factor,
		\begin{align*}
			\mu(\mathcal{A}_{FT2}\mathcal{D}_L)(f\otimes\bar g)(x,\xi)&=\mathcal{F}_2\mathfrak{T}_{L}(f\otimes \bar{g})(x,\xi)=\int_{\rd} (f\otimes \overline{g})(L(x,y))e^{-2\pi i\xi\cdot y}dy\\
			&=\int_{\rd} f(x+(I_{d\times d}-A_{11})y)\overline{g(x-A_{11}y)}e^{-2\pi i\xi \cdot y}dy.
		\end{align*}
		For $z=(x,\xi)$ and $V_C$ as in the statement, we have by \eqref{muVc}
		\[
			\mu(V_C)F(z)=e^{\pi i Cz\cdot z}F(z)=e^{i\pi [(C_{11}x+C_{12}\xi)\cdot x+(C_{12}^T x+C_{22}\xi)\cdot \xi]}F(z), \quad F\in L^2(\mathbb{R}^{2d}).
		\]	
		Furthermore, formula (\ref{WaVg}) applied to the symplectic matrix $\mathcal{A}_{FT2}\mathcal{D}_L$ ($L$ as in the statement, see \cite[Theorem 2.27]{CR2022}, where the formula was obtained in this particular case) tells that, up to a unitary constant, 
		\begin{align*}
			W_\mathcal{A}(f,g)(z)&=e^{\pi i (Cz)\cdot z}|\det(I_{d\times d}-A_{11})|^{-1}e^{2\pi i((I_{d\times d}-A_{11}^T)^{-1}\xi)\cdot x}\\
			&\qquad\qquad \times V_{\tilde{g}}f(A_{11}^{-1}x,(I_{d\times d}-A_{11}^T)^{-1}\xi),
		\end{align*}
		for $f,g\in L^2(\mathbb{R}^d)$ and $\tilde g$ being as in the statement.
	\end{proof}
	
	As a consequence, we extend \cite[Theorem 2.28]{CR2022} to all Wigner-decomposable matrices. 
	
	\begin{corollary}\label{lemmaE}
		Under the notation of Theorem \ref{charWAtotWDec}, the following statements are equivalent:\\
		(i) $V_C\mathcal{A}_{FT2}\mathcal{D}_L$ is shift-invertible,\\
		(ii) $\mathcal{A}_{FT2}\mathcal{D}_L$ is shift-invertible,\\
		(iii) $L$ is right-regular.
	\end{corollary}
	\begin{proof}
		The equivalence $(i)\Leftrightarrow(ii)$ is proved in Corollary \ref{corReprL}. The equivalence $(ii)\Leftrightarrow(iii)$ follows from Theorem \ref{charWAtotWDec}, which gives:
		\[
			W_\mathcal{A}(f,g)(x,\xi)=e^{i\pi [(C_{11}x+C_{12}\xi)\cdot x+(C_{12}^T x+C_{22}\xi)\cdot \xi]} W_{\mathcal{A}_{FT2}\mathcal{D}_L}(f,g)(x,\xi),
		\]
		so that: 
		\begin{equation}\label{eqModuli}
			|W_\mathcal{A}(f,g)(x,\xi)|=|W_{\mathcal{A}_{FT2}\mathcal{D}_L}(f,g)(x,\xi)|.
		\end{equation}
		This gives
		\[
			|W_\mathcal{A}(\pi(w)f,g)|=|W_{\mathcal{A}_{FT2}\mathcal{D}_L}(\pi(w)f,g)|, \quad \forall w\in\rdd,
		\]
		which proves the claim.
	\end{proof}
	
	\begin{corollary}
		Let $\mathcal{A}\in Sp(2d,\mathbb{R})$ be Wigner-decomposable, with matrix $L$ right-regular. Then, for any $g\in \mathcal{S}(\mathbb{R}^d)\setminus\{0\}$, $0<p,q\leq\infty$,
		\[
			f\in M^{p,q}_{v_s}(\rd)\quad\Longleftrightarrow\quad W_\mathcal{A}(f,g)\in L^{p,q}_{v_s}(\rdd).
		\]
		For $1\leq p,q\leq\infty$, the window $g$ can be chosen in $M^1_{v_s}(\rd)$.
	\end{corollary}
	\begin{proof}
	With the same notations as Theorem \ref{charWAtotWDec}, write $\mathcal{A}=V_C\mathcal{A}_{FT2}\mathcal{D}_L$, with $L$ is right-invertible. By (\ref{eqModuli}),
		\[	W_\mathcal{A}(f,g)\in L^{p,q}_{v_s}(\rdd)\quad\Longleftrightarrow \quad W_{\mathcal{A}_{FT2}\mathcal{D}_L}(f,g)\in L^{p,q}_{v_s}(\rdd).
		\]
		 By Corollary \ref{lemmaE}, $\mathcal{A}_{FT2}\mathcal{D}_L$ is a covariant (see Subsection \ref{subsec:ss42} below), shift-invertible matrix. Then the claim follows from \cite[Theorem 2.28]{CR2022}. 
	\end{proof}

\subsection{Covariance}\label{subsec:ss42}
	According to \cite[Proposition 2.10]{CR2022}, for a given symplectic matrix $\mathcal{A}$, the {metaplectic Wigner distribution} $W_\mathcal{A}$ satisfies
	\begin{equation}\label{covWA}
		W_\mathcal{A}(\pi(z)f,\pi(z)g)=T_zW_\mathcal{A}(f,g), \ \ \ \ f,g\in\mathcal{S}(\mathbb{R}^d), \ \ \ z\in\rdd,
	\end{equation}
	if and only if $\mathcal{A}$ has block decomposition
	\begin{equation}\label{blockAcov}
		\mathcal{A}=\begin{pmatrix}
			A_{11} & I_{d\times d}-A_{11} & A_{13} & A_{13}\\
			A_{21} & -A_{21} & I_{d\times d}-A_{11}^T & -A_{11}^T\\
			0_{d\times d} & 0_{d\times d} & I_{d\times d} & I_{d\times d}\\
			-I_{d\times d} & I_{d\times d} & 0_{d\times d} & 0_{d\times d}
		\end{pmatrix},
	\end{equation}
	with $A_{13}=A_{13}^T$ and $A_{21}=A_{21}^T$. We refer to such matrices as to \textbf{covariant matrices} and to property (\ref{covWA}) as to the \textbf{covariance} property of $W_\mathcal{A}$. It was proved in \cite{CR2022} that a covariant matrix with block decomposition (\ref{blockAcov}) is totally Wigner-decomposable if and only if $A_{21}=A_{13}=0_{d\times d}$. Moreover, if
	\begin{equation}
	\label{covBCohen}
	B_\mathcal{A}:=\begin{pmatrix}
		A_{13} & \frac{1}{2}I_{d\times d}-A_{11}\\
		\frac{1}{2}I_{d\times d} & -A_{21}
	\end{pmatrix},
	\end{equation}
	
	and $W$ is the classical Wigner distribution, the following result holds: 
	
	\begin{theorem}\label{ThmCohen}
		Let $\mathcal{A}\in Sp(2d,\mathbb{R})$ be a covariant matrix in the form (\ref{blockAcov}). Then, 
		\begin{equation}\label{CohenCovWA}
		W_\mathcal{A}(f,g)=W(f,g)\ast \Sigma_\cA, \quad f,g\in\mathcal{S}(\mathbb{R}^d),
		\end{equation}
		where
		\begin{equation}\label{CohenSymbol}
			\Sigma_\cA(z)=\mathcal{F}^{-1}(e^{-\pi i\zeta\cdot B_{\mathcal{A}}\zeta})\in\mathcal{S}'(\mathbb{R}^{2d}),
		\end{equation}
		and $B_\mathcal{A}$ defined as in (\ref{covBCohen}).
	\end{theorem}
	Recalling our chirp function in \eqref{chirp}, the equality in \eqref{CohenSymbol} can be rewritten as
	\begin{equation}\label{Cohen-chirp}
\Sigma_\cA=\mathcal{F}^{-1}\Phi_{- B_{\mathcal{A}}}.
	\end{equation}
	If a time-frequency representation $Q(f,g)$ satisfies 
	\[	
		Q(f,g)=W(f,g)\ast \Sigma
	\]
	for some $\Sigma\in\mathcal{S}'(\rdd)$, we say that $Q$ belongs to the \textbf{Cohen class}, cf. \cite{Cohen1,Cohen2}.\\
	
%

	Theorem \ref{ThmCohen} sheds light on the importance of covariant matrices in the context of time-frequency analysis, stating that $\mathcal{A}\in Sp(2d,\mathbb{R})$ is covariant if and only if $W_\mathcal{A}$ belongs to the Cohen class. The following result shows that covariant matrices are exactly those that decompose as the product of symplectic matrices $V_C^T$, $\mathcal{A}_{FT2}$ and $\mathcal{D}_L$ for some $d\times d$ symmetric matrix $C$ and $L\in GL(2d,\rd)$.
	
	\begin{theorem}\label{charWAcovariant}
		Let $\mathcal{A}\in Sp(2d,\mathbb{R})$ be covariant with block decomposition (\ref{blockAcov}). Then, 
		\begin{equation}\label{decompCovMatrices}
			\mathcal{A}=V_C^T\mathcal{A}_{FT2}\mathcal{D}_L,
			\end{equation}
			 where
		\begin{equation}\label{defVcELcov}
			C=\begin{pmatrix}
			A_{13} & 0_{d\times d}\\
		0_{d\times d} & -A_{21}
			\end{pmatrix}		
\ \ \ \text{and} \ \ \ 
			L=\begin{pmatrix}
				I_{d\times d} & I_{d\times d}-A_{11}\\
				I_{d\times d} & -A_{11}
			\end{pmatrix}.
		\end{equation}
		As a consequence, up to a phase factor, for all $f,g\in \mathcal{S}(\mathbb{R}^d)$,
		\begin{equation}\label{forInPiu}
		W_\mathcal{A}(f,g)(x,\xi)=\int_{\rd} [\mathcal{F}(\Phi_{- A_{13}})\ast (f\otimes \bar{g})(L(\cdot,\eta))](x)\Phi_{A_{21}} (\eta)e^{-2\pi i\xi\cdot \eta}d\eta.
		\end{equation}
		
		 In particular, if $A_{13}=0_{d\times d}$, then
		\begin{equation}\label{forInPius}
			W_\mathcal{A}(f,g)(x,\xi)=\int_{\rd}  f(x+(I_{d\times d}-A_{11})\eta)\overline{g(x-A_{11}\eta)}\Phi_{A_{21}}(\eta)e^{-2\pi i\xi\cdot \eta}d\eta.
		\end{equation}
	\end{theorem}
	\begin{proof}
Equality \eqref{decompCovMatrices} is a straightforward computation. Now, using  \eqref{muVctrans},
\[\begin{split}
	W_\cA(f,g)(x,\xi)&=\mu(\cA)(f\otimes\bar g)(x,\xi)=\mu(V_C^T)\mu(\mathcal{A}_{FT2})\mu(\mathcal{D}_{L})(f\otimes\bar g)(x,\xi)\\
		&=\mathcal{F}(\Phi_{-C})\ast\Big(\int_{\rd}(f\otimes \bar{g})(L(\cdot,\eta))e^{-2\pi i\eta\cdot(\bullet)}d\eta\Big)(x,\xi)\\
	&=\int_{\rdd}\mathcal{F}(\Phi_{-C})(x-y,\xi-\omega)\Big(\int_{\rd}(f\otimes \bar g)(L(y,\eta))e^{-2\pi i\eta\cdot \omega}d\eta\Big)dyd\omega\\
	&=\int_{\rdd}\Big(\int_{\rdd}e^{-i\pi[A_{13}u\cdot u-A_{21}v\cdot v]}e^{-2\pi i[(x-y)\cdot u+(\xi-\omega)\cdot v]}dudv\Big)\\
	&\qquad\qquad\times\Big(\int_{\rd}(f\otimes \bar g)(L(y,\eta))e^{-2\pi i\eta\cdot \omega}d\eta\Big)dyd\omega.
\end{split}
\]
Observe that
\[
	\begin{split}
		\int_{\rdd}e^{i\pi A_{21}v\cdot v}e^{-2\pi i\xi\cdot v}e^{2\pi i\omega\cdot (v-\eta)}dvd\omega=e^{\pi iA_{21}\eta\cdot\eta}e^{-2\pi i \xi\cdot \eta},
	\end{split}
\]
so that
\[
	\begin{split}
	W_\cA(f,g)(x,\xi)&=\int_{\mathbb{R}^{d}}\Big(\int_{\rdd}e^{-i\pi A_{13}u\cdot u}e^{2\pi i u\cdot(x-y)}(f\otimes\bar g)(L(y,\eta))dudy\Big)e^{\pi iA_{21}\eta\cdot\eta}e^{-2\pi i\xi\cdot\eta}d\eta.
	\end{split}
\]
Next, we apply
\[\begin{split}
	\int_{\rdd}\varphi_1(u)\varphi_2(y)e^{2\pi iu\cdot x}e^{-2\pi iu\cdot y}dudy&=\int_{\rd}\Big(\int_{\rd}\varphi_2(y)e^{-2\pi iu\cdot y}dy\Big)\varphi_1(u)e^{2\pi iu\cdot x}du\\
	&=\int_{\rd}\hat\varphi_2(u)\varphi_1(u)e^{2\pi ix\cdot u}du=\mathcal{F}^{-1}(\varphi_1\hat \varphi_2)(x)\\
	&=(\mathcal{F}^{-1}(\varphi_1)\ast\varphi_2)(x)
\end{split}
\]
to the inner integral, to get \eqref{forInPiu}.
	\end{proof}
	
	
		\begin{remark}
		Theorem \ref{charWAcovariant} states that the class of covariant symplectic matrices is invariant with respect to left-multiplication by matrices $V_C^T$. Equivalently, the class of metaplectic Wigner distributions associated to covariant matrices is invariant with respect to convolutions by kernels in the form $\Phi_C$, $C$ $d\times d$ real symmetric matrix.
	\end{remark}
	
	\begin{remark}\label{remA1321}
		Theorem \ref{charWAcovariant} clarifies the roles that the blocks $A_{13}$ and $A_{21}$ have in Wigner metaplectic operators associated to covariant matrices. The block $A_{13}$ appears in the convolution factor $\mathcal{F}(\Phi_{-A_{13}})(\cdot)$ and acts on $(f\otimes \bar g)\circ L(\cdot,\eta)$, whereas $A_{21}$ produces the phase factor $\Phi_{A_{21}}$. 
	\end{remark}
	
	As we pointed out, covariant matrices play a key part in the theory of pseudodifferential operators, as they belong to the Cohen class. In the following result we prove an explicit integral formula for metaplectic pseudodifferential operators associated to covariant matrices.
	
	\begin{proposition}
		Let $\mathcal{A}\in Sp(2d,\mathbb{R})$ be a covariant matrix with decomposition in \eqref{decompCovMatrices}. Then, for every $f\in\mathcal{S}(\rd)$ and  $a\in\mathcal{S}'(\rdd)$, up to a phase factor,
		\begin{equation}\label{formulaOpA}
			Op_\mathcal{A}(a)f(x)=\int_{\rdd}(\cF(\Phi_{C})\ast a)(A_{11}x+(I_{d\times d}-A_{11})y, \xi)f(y)e^{2\pi i\xi\cdot(x-y)}dyd\xi,
		\end{equation}
		where the chirp function $\Phi_C$ is defined in \eqref{chirp}.
	\end{proposition}
	\begin{proof}
		We use the expression of $W_{\cA}$ and Theorem \ref{charWAcovariant}. Namely, for every $f,g\in\mathcal{S}(\rd)$, $a\in\mathcal{S}'(\rdd)$, up to a unitary factor,
		\begin{align*}
		\langle Op_\mathcal{A}(a)f,g\rangle&=\langle a,W_\mathcal{A}(g,f)\rangle=\langle a, \mu(V_C^T\mathcal{A}_{FT2}\mathcal{D}_L)(g\otimes \bar f)\rangle=\langle \mathfrak{T}_{L^{-1}}\mathcal{F}_2^{-1}\mu(V_C^{-T})a,g\otimes \bar f\rangle,
		\end{align*}
		where we used $\mu(V_C^T)^{\ast}=\mu(V_C^{-T})$. Since $|\det(L)|=1$, we can write
		\begin{align*}
			\langle Op_\mathcal{A}(a)f,g\rangle&=\int_{\rdd}\mathcal{F}_2^{-1}\mu(V_C^{-T})a(L^{-1}(x,y))\overline{g(x)} f(y)dxdy\\
			&=\int_{\rd}\left(\int_{\rd}\mathcal{F}_2^{-1}\mu(V_C^{-T})a(L^{-1}(x,y)) f(y)dy\right) \overline{g(x)} dx\\
			&=\left\langle \int_{\rd}\mathcal{F}_2^{-1}\mu(V_C^{-T})a(L^{-1}(x,y)) f(y)dy,g\right\rangle,
		\end{align*}
		where the integrals must be intended in the weak sense. Hence,
		\begin{equation}\label{defOpAimplicita}
			Op_{\mathcal{A}}f(x)=\int_{\rd}((\mathcal{F}_2^{-1}\mu(V_C^{-T}))a)(L^{-1}(x,y)) f(y)dy.
		\end{equation}
		Using 
		\begin{equation}\label{invdiLcov}
					L^{-1}=\begin{pmatrix}
						A_{11} & I_{d\times d}-A_{11}\\
						I_{d\times d} & -I_{d\times d}
					\end{pmatrix},
				\end{equation}
				we compute
\begin{align*}
			Op_{\mathcal{A}}f(x)&=\int_{\rd}\left(\int_{\rd} (\mu(V_C^{-T})a)(A_{11}x+(I_{d\times d}-A_{11})y, \xi)e^{2\pi i\xi\cdot(x-y)}d\xi\right)f(y)dy\\
			&=\int_{\rdd}(\mu(V_C^{-T})a)(A_{11}x+(I_{d\times d}-A_{11})y, \xi)e^{2\pi i\xi\cdot(x-y)}f(y)dyd\xi\\
			&=\int_{\rdd}(\cF(\Phi_{C})\ast a)(A_{11}x+(I_{d\times d}-A_{11})y, \xi)e^{2\pi i\xi\cdot(x-y)}f(y)dyd\xi
		\end{align*}
		where in the last-but-one row we used the expression of $\mu(V_C^{-T})$ computed in \eqref{vcT}. 
	\end{proof}
	
	\begin{remark}
		With the same fashion of Remark \ref{remA1321}, we stress that (\ref{formulaOpA}) sheds light on the role of the matrix $V_C^T$, in the decomposition of a covariant matrix $\cA$, on the pseudodifferential operator with quantization given by $\cA$. Basically, it produces the chirp $\mathcal{F}\Phi_C$ which acts on the symbol $a$ via convolution. 
	\end{remark}

	To study the solution $u=u(x,t)$ to the Schr\"odinger equation in \eqref{C12} we need to know information about his projection $\chi_t$ in \eqref{e7}.
		\begin{lemma}
		Consider a covariant matrix $\mathcal{A}\in Sp(2d,\mathbb{R})$ having block decomposition (\ref{blockAcov}) and related matrix $B_\cA$ in \eqref{covBCohen}. For $\chi_t$, $t\in\bR$, in \eqref{e7}, assume that its inverse  $\chi_t^{-1}\in Sp(d,\mathbb{R})$,  has the $d\times d$ block decomposition
		\[
			\chi_t^{-1}=\begin{pmatrix}
				X_t & Y_t\\ W_t & Z_t
			\end{pmatrix}.
		\]
		Let ${B}_{\mathcal{A}_t}=\chi_t^{-T}B_\mathcal{A}\chi_t^{-1}$ and $\mathcal{A}_t\in Sp(2d,\mathbb{R})$ be the symplectic matrix associated to ${B}_{\mathcal{A}_t}$. Then, $\mathcal{A}_t$ is the covariant matrix having block decomposition 
		\begin{equation}\label{blockAtcov}
		\mathcal{A}_t=\begin{pmatrix}
			A_{t,11} & I_{d\times d}-A_{t,11} & A_{t,13} & A_{t,13}\\
			A_{t,21} & -A_{t,21} & I_{d\times d}-A_{t,11}^T & -A_{t,11}^T\\
			0_{d\times d} & 0_{d\times d} & I_{d\times d} & I_{d\times d}\\
			-I_{d\times d} & I_{d\times d} & 0_{d\times d} & 0_{d\times d}
		\end{pmatrix},
	\end{equation}
		with
		\begin{align*}
			&A_{t,11}=-W_t^TY_t-X_t^T[A_{13}Y_t-A_{11}Z_t]+W_t^T[A_{11}^TY_t+A_{21}Z_t],\\
			&A_{t,13}=X_t^TW_t+X_t^T[A_{13}X_t-A_{11}W_t]-W_t^T[A_{11}^TX_t+A_{21}W_t],\\
			&A_{t,21}=-Z_t^TY_t-Y_t^T[A_{13}Y_t-A_{11}Z_t]+Z_t^T[A_{11}^TY_t+A_{21}Z_t].
		\end{align*}
	\end{lemma}
	\begin{proof}
		Plugging ${B}_{\mathcal{A}_t}=\chi_t^{-T}B_\mathcal{A}\chi_t^{-1}$ into the block decomposition (\ref{covBCohen}) for ${B}_\mathcal{A}$ and using the symplectic properties 
		\begin{align*}
			&W_t^TX_t=X_t^TW_t,\\
			&Z_t^TY_t=Y_t^TZ_t,\\
			&Z^T_tX_t-Y_t^TW_t=I_{d\times d}
		\end{align*}
		of $\chi_t^{-1}$ we get
		\begin{equation}\label{Bevol}
			{B}_{\mathcal{A}_t}=\begin{pmatrix}
				A_{t,13} & \frac{1}{2}I_{d\times d}-A_{t,11}\\
				\frac{1}{2}I_{d\times d}-A_{t,11}^T & -A_{t,21}
			\end{pmatrix},
		\end{equation}
		where $A_{t,11},A_{t,13}$ and $A_{t,21}$ are defined as in the assertion. Since the covariance of $\mathcal{A}$ is inherited by $\mathcal{A}_t$, we have that these blocks are exactly the ones defining the block decomposition of $\mathcal{A}_t$ as a covariant matrix. 
	\end{proof}

	We can now express the phase-space concentration of the solution $u(x,t)$ to the  free particle equation in terms of $\cA$-Wigner distribution.
	\begin{example}[The free particle]\label{E2}
	We shall prove the  formula originally announced in Part I, see Example $4.9$ in \cite{CR2021}, formula $(126)$ therein (see also formula (108) in \cite{CR2022}).
		
	In Example $4.9$ in \cite{CR2021} we computed the
		$\tau$-Wigner  of the solution  $u(t,x)$ to the Cauchy problem of the  free particle equation:
		\begin{equation}\label{C26-0} 
		\begin{cases}
		i\partial_t u+\Delta u=0,\\
		u(0,x)=u_0(x),
		\end{cases}
		\end{equation}
		with $(t,x)\in\bR\times\bR^d$, $d\geq1$. Namely,  we obtained that 
		\begin{equation}\label{FP1}
		W_\tau u(t,x,\xi)=W_{\cA_{\tau,t}} u_0(x-4\pi t \xi,\xi),
		\end{equation}
		where the representation 	$W_{\cA_{\tau,t}}$ is of Cohen class:
		\begin{equation}\label{E103}
			W_{\cA_{\tau,t}}f=Wf\ast\Sigma_{\tau,t},
		\end{equation}
		with kernel
		\begin{equation*}
		\Sigma_{\tau,t}(x,\xi)=\Sigma_\tau(x+4\pi t \xi, \xi),
		\end{equation*}
	where, for $\tau\not=1/2$ (Wigner case), the $\tau$-kernel is given by
		\begin{equation*}
		\Sigma_\tau \phas= 	\frac{2^d}{|2\tau -1|^d}e^{2\pi i\frac{2}{2\tau -1}x \xi}.
		\end{equation*}
	The matrix  $B_{\bf{A}_\tau}$ in \eqref{covBCohen} can be computed as
	\begin{equation}\label{aggiunta9}
	B_{\bf{A}_\tau}= \left(\begin{array}{cc}
	0_{d\times d} & (\tau-\frac12)I_{d\times d}\\
	(\tau-\frac12)I_{d\times d}&0_{d\times d}
	\end{array}\right),
	\end{equation}	
	and by \eqref{Bevol} (see also Proposition 4.4 in \cite{CR2022}),
	$$B_{\cA_{\tau,t}}=\chi_t^{-T} B_{\bf{A}_\tau}\chi_t^{-1}=\left(\begin{array}{cc}
	0_{d\times d}& (\tau-\frac12)I_{d\times d}\\
	(\tau-\frac12)I_{d\times d} &  (4\pi t)(1-2\tau)I_{d\times d}
	\end{array}\right).$$	
		The representation \eqref{E103} can be equivalently written as (cf. \eqref{Cohen-chirp})
		$$	W_{\cA_{\tau,t}}f=Wf\ast\cF^{-1}\Phi_{\cA_{\tau,t}}.
		$$
Hence,  the  ${\cA_{\tau,t}}$-Wigner representation  computed in \eqref{forInPius} with $$A_{t,13}=0_{d\times d}, \quad A_{t,11}=(1-\tau) I_{d\times d}, \quad A_{t,21}= -(4\pi t)(1-2\tau)I_{d\times d}$$ becomes
	\begin{equation*}
W_{\cA_{\tau,t}}(f,g)(x,\xi)=\int_{\rd}  f(x+\tau\eta)\overline{g(x-(1-\tau)\eta)}e^{-2\pi i(\xi\cdot \eta+2\pi t(1-2\tau)\eta^2)}d\eta,
\end{equation*}
as desired.
\end{example}

\section{Continuity on modulation spaces}
For many quantizations, $Op_\cA$ is an \textit{integral superposition} of time-frequency shifts. Stated differently, these fundamental operators of time-frequency analysis represent the building blocks of pseudodifferential operators. Concretely, the Weyl quantization of a pseudodifferential operator with symbol $a\in\mathcal{S}'(\rdd)$ is given by
	\[
		Op_w(a)=\int_{\rdd}\hat a(\eta,-z)e^{-i\pi\eta\cdot z}\pi(z,\eta)dzd\eta.
	\]
	On the other hand, if $f\in M^{p,q}_m$ for some $m\in\mathcal{M}_{v_s}$ and $0<p,q\leq\infty$, then $\pi(z,\eta)f\in M^{p,q}_m$ for all $z,\eta\in\rd$. This turns out to be one one of the main reasons why modulation spaces appear in the theory of pseudodifferential operators.
	
	In this section, we use the results in the first part of this paper to investigate the continuity properties of metaplectic pseudodifferential operators on modulation spaces. Since weighted modulation spaces measure the phase-space concentration of signals, as well as their decay properties, an investigation of their continuity on these spaces reveals how the time-frequency concentration of signals changes when a pseudodifferential operator is applied. \\
	
	The first result we present involves the explicit expression of the symbol $b:=(a\otimes 1)\circ\cA^{-1}$, as in the equality (\ref{commOpWParticolare2}) above, when $\cA$ is totally Wigner-decomposable or covariant.
	
	\begin{proposition}
		Consider $\mathcal{A}\in Sp(2d,\mathbb{R})$, $a\in\mathcal{S}'(\rdd)$ and $b=\sigma\circ\mathcal{A}^{-1}$, with $\sigma=a\otimes 1$ as defined in (\ref{defsigma}).  For every $x,\xi,u,v\in\rd$ we can state:\\
		(i) if $\mathcal{A}$ is totally Wigner-decomposable with block decomposition as in Proposition \ref{explA}, then
		\begin{equation}\label{expressbsymb}
			b(x,\xi,u,v)=a(A_{33}^Tx-A_{23}^Tv,-A_{41}^T\xi+A_{11}^Tu);
		\end{equation}
		(ii) if $\mathcal{A}$ is covariant with block decomposition as in (\ref{blockAcov}), then
		\begin{equation}\label{expressbsymb2}
			b(x,\xi,u,v)=a(x-A_{13}u+(A_{11}-I)v,\xi+A_{11}^Tu+A_{21}v).
		\end{equation}
	\end{proposition}
	\begin{proof}
		The proof follows by the straightforward calculation of 
	\begin{equation}\label{E1}
			\mathcal{A}^{-1}(x,\xi,u,v)^T.
		\end{equation}
		Namely, to get \eqref{expressbsymb} one applies \eqref{E1} with $\mathcal{A}^{-1}$ as in (\ref{Ainv}), whereas \eqref{expressbsymb2} is obtained applying \eqref{E1} with 
		\[
			\mathcal{A}^{-1}=\begin{pmatrix}
				I_{d\times d} & 0_{d\times d} & -A_{13} & A_{11}-I_{d\times d}\\
				I_{d\times d} & 0_{d\times d} & -A_{13} & A_{11}\\
				0_{d\times d} & I_{d\times d} & A_{11}^T & A_{21}\\
				0_{d\times d} & -I_{d\times d} & I_{d\times d}-A_{11}^T & -A_{21}
			\end{pmatrix}.
		\]
	\end{proof}	
	
	For $a\in\cS(\rdd)$, define $\sigma:=a\otimes1$ as in (\ref{defsigma}), and
	\begin{align*}
		\tilde\sigma(r,y,\rho,\eta)=1_{(r,\rho)}\otimes\bar a(y,-\eta).
	\end{align*}
	For $\mathcal{A}\in Sp(2d,\mathbb{R})$ we set
	\begin{align}
	\label{defb}
		&b(x,\xi,u,v)=(\sigma\circ\mathcal{A}^{-1})(x,\xi,u,v),\\
		\label{deftb}
		&\tilde b(x,\xi,u,v)=(\tilde\sigma\circ\mathcal{A}^{-1})(x,\xi,u,v),\\
		\label{defc}
		& c(x,\xi,u,v)=b(x,\xi,u,v)\tilde b(x,\xi,u,v).
	\end{align}
	
	The following result extends Lemma 5.1 in \cite{CR2021} to general symplectic matrices.	
	\begin{lemma}\label{lemma18}
		Let $\mathcal{A}\in Sp(2d,\mathbb{R})$, $a\in M^{\infty,q}_{1\otimes v_s}(\rdd)$, $0<q\leq\infty$ and $s\geq0$. Let $b,\tilde b$ and $c$ be defined as in (\ref{defb}), (\ref{deftb}) and (\ref{defc}), respectively. Then $b, \tilde b, c$ are in $M^{\infty,q}_{1\otimes v_s}(\mathbb{R}^{4d})$.
	\end{lemma}
	\begin{proof}
	The proof that $b$ and $\tilde b$ are in $M^{\infty,q}_{1\otimes v_s}(\rdd)$ is analogous to that of \cite[Lemma 5.1]{CR2021}. In fact, observe $1_{(y,\eta)}$ is in $M^{\infty,q}_{1\otimes v_s}(\rdd)$ for every $0<q\leq\infty$ and $s\geq 0$. 
	For $ c=b\tilde b$, if $q\geq1$ we use the product properties for modulation spaces in \cite[Proposition 2.4.23]{Elena-book}, the quasi-Banach case $0<q<1$ is contained in \cite{GCFZ19}.
	\end{proof}
	
	Recall the following boundedness result for Weyl quantization, see \cite[Theorem 14.5.6]{grochenig}, \cite{Toftweight2004} and \cite[Theorem 3.1]{ToftquasiBanach2017}.
	\begin{proposition}\label{lemmaGrochenig}
		If $0< p,q,r \leq\infty$ with $r=\min\{1,p,q\}$, $s\in\bR$, $\sigma\in M^{\infty,r}_{1\otimes v_{|s|}}(\rdd)$, then {$Op_w(\sigma):\mathcal{S}(\mathbb{R}^d)\to\mathcal{S}'(\mathbb{R}^d)$} extends to a bounded operator on $\mathcal{M}_{v_s}^{p,q}(\mathbb{R}^d)$.
	\end{proposition}
	
	We generalize Proposition \ref{lemmaGrochenig} to {metaplectic pseudodifferential operators}:
	
	\begin{theorem}\label{thmGrochenig}
		Consider $\mathcal{A}\in Sp(2d,\mathbb{R})$ a covariant matrix such that $B_\mathcal{A}$ in (\ref{covBCohen}) is invertible. For $0< p,q\leq\infty$, set $r=\min\{1,p,q\}$. If $a\in M^{\infty,r}_{1\otimes v_s}$, $s\geq0$, then $Op_\mathcal{A}(a):\mathcal{S}(\rd)\to\mathcal{S}'(\rd)$ extends to a bounded operator on $\mathcal{M}_{v_s}^{p,q}(\mathbb{R}^d)$.
	\end{theorem}
	\begin{proof}
		By \cite[Proposition 3.3]{CR2022}, $\mathcal{F}^{-1}\Phi_{B_\cA}\in M^{r,\infty}_{v_s\otimes 1}(\rdd)$ for every $s\geq0$, $0<r\leq\infty$. Since $W_\mathcal{A}$ belongs to the Cohen class, for every $f,g\in\mathcal{S}(\rd)$,
		\begin{align*}
			\langle Op_\mathcal{A}(a)f,g\rangle&=\langle a,W_\mathcal{A}(g,f)\rangle=\langle a, W(g,f)\ast \mathcal{F}^{-1}\Phi_{-B_\cA}\rangle\\
			&=\langle \hat a , \mathcal{F}(W(g,f))e^{-i\pi\zeta\cdot B_\mathcal{A}\zeta}\rangle=
			\langle \hat a e^{i\pi\zeta\cdot B_\mathcal{A}\zeta}, \mathcal{F}(W(g,f))\rangle\\
			&=\langle a\ast \mathcal{F}^{-1}\Phi_{B_\cA},W(g,f)\rangle=\langle Op_w(a\ast \mathcal{F}^{-1}\Phi_{B_\cA})f,g\rangle.	
		\end{align*}
		By \cite[Proposition 3.1]{BC2021} 
		\[
		\Vert a\ast\mathcal{F}^{-1}\Phi_{B_\cA}\Vert_{M^{\infty,r}_{1\otimes v_s}}\lesssim\Vert a\Vert_{M^{\infty,r}_{1\otimes v_s}}\Vert \mathcal{F}^{-1}\Phi_{B_\cA}\Vert_{M^{r,\infty}_{v_s\otimes 1}}.
		\]
		The assertion follows from \cite[Theorem 3.1]{ToftquasiBanach2017}.
	\end{proof}
	
We conclude this section by showing the validity of relations (\ref{commOpWParticolare}) on modulation spaces. 
	
	\begin{theorem}\label{5.4}
		Consider $\mathcal{A}\in Sp(2d,\mathbb{R})$, $0< p\leq \infty$, $a\in M^{\infty,r}_{1\otimes v_s}(\rdd)$, $s\geq0$, $r=\min\{1,p\}$, and $b$, $\tilde b$, $c$ defined as in (\ref{defb}), (\ref{deftb}) and (\ref{defc}), respectively. For $f,g\in\mathcal{M}^p_{v_s}(\rd)$, the following identities hold in $\cM^{p}_{v_s}(\rdd)$: 
			\begin{align}
		W_\mathcal{A}(Op_{w,2d}(a)f,g)=Op_{w,4d}(b)W_\mathcal{A}(f,g),\label{A4}\\
		W_\mathcal{A}(f,Op_{w,2d}(a)g)=Op_{w,4d}(\tilde b)W_\mathcal{A}(f,g),\label{A5}\\
		W_\mathcal{A}(Op_{w,2d}(a)f)=Op_{w,4d}(c)W_\mathcal{A}(f). \label{A6}
		\end{align}
	\end{theorem}
	\begin{proof}
	If $f\in\mathcal{M}^{p}_{v_s}(\rd)$, then $Op_w(a)f\in \mathcal{M}^p_{v_s}(\rd)$ 	by \cite[Theorem 3.1]{ToftquasiBanach2017}. Hence, \cite[Theorem 2.15]{CR2022} says that $$W_\mathcal{A}(f), W_\mathcal{A}(f,g), W_\mathcal{A}(Op_w(a)f,g), W_\mathcal{A}(f,Op_w(a)g), W_\mathcal{A}(Op_w(a)f)\in\mathcal{M}^p_{v_s}(\rdd).$$ 
	Similarly, by Lemma \ref{lemma18}, the symbols $b,\tilde{b}$ and $c$ are in $M^{\infty,q}_{1\otimes v_s}(\bR^{4d})$ and the right-hand sides of formulas \eqref{A4}, \eqref{A5} and \eqref{A6} are in $\cM^p_{v_s}(\rdd)$. The equalities \eqref{A4}, \eqref{A5} and \eqref{A6} are obtained by using the same pattern as in the proof of \cite[Theorem 5.1]{CR2021}, namely replacing the symplectic matrix $\mathcal{A}_\tau$ with a general $\mathcal{A}\in Sp(2d,\bR)$.
	\end{proof}

\begin{remark}
	Observe that the previous result extends \cite[Theorem 5.1]{CR2021} to the quasi-Banach setting $0<p<1$.
\end{remark}

\section{Algebras of generalized metaplectic operators}
In this section we introduce (quasi-)algebras of FIOs which extend the ones in \cite{CGNRJMPA,CGNRJMP2014}. 

Recall the definition of a Gabor frame. Given a lattice in the phase-space $\Lambda=A\zdd$,  with $A\in GL(2d,\R)$, and a non-zero window function $g\in L^2(\rd)$,  a \emph{Gabor system} is the sequence: $$\cG(g,\Lambda)=\{\pi(\lambda)g:\
\lambda\in\Lambda\}.$$
A Gabor system $\cG(g,\Lambda)$   becomes
a Gabor frame if there exist
constants $0<A\leq B$ such that
\begin{equation}\label{gaborframe}
A\|f\|_2^2\leq\sum_{\lambda\in\Lambda}|\langle f,\pi(\lambda)g\rangle|^2\leq B\|f\|^2_2,\qquad \forall f\in L^2(\rd).
\end{equation}
Given a Gabor frame $\cG(g,\Lambda)$,  the \emph{Gabor matrix} of a linear continuous operator $T$ from $\cS(\rd)$ to $\cS'(\rd)$ is 
\begin{equation}\label{unobis2s} \langle T \pi(z)
g,\pi(u)g\rangle,\quad z,u\in \rdd.
\end{equation}

\emph{Our goal}: controlling the Gabor matrix of a metaplectic operator $T$ (or more general one) related to the symplectic matrix $\chi\in Sp(d,\bR)$ by 
$$h(\mu-\chi\lambda),\quad\lambda,\mu\in\Lambda,$$
where   $h$ is a sequence leaving  in a suitable (quasi-)algebra with respect to convolution. 

The algebras already studied in \cite{CGNRJMPA,CGNRJMP2014} where $\ell^1(\Lambda)$ and $\ell^\infty_{v_s}(\Lambda)$, $s>2d$. Here we extend to the quasi-algebras $\ell_{v_s}^q(\Lambda)$, $0<q<1$, $s\geq0$, enjoying the convolution property: 
$$\ell_{v_s}^q(\Lambda)\ast\ell_{v_s}^q(\Lambda)\hookrightarrow\ell_{v_s}^q(\Lambda),\quad 0<q<1.$$

Recall that the  Wiener amalgam spaces
 $W(C,L^p_{v_s})(\rdd)$ is defined in \eqref{Wiener-space} and the class $FIO(\chi,q,v_s)$ is defined in Definition \ref{def1.1}.


The union
\[
FIO(Sp(d,\R),q,v_s)=\bigcup_{\chi\in Sp(d,\R)} FIO(\chi,q,v_s)
\]
is called the class of \emph{generalized metaplectic operators}. Similarly to \cite[Proposition 3.1]{CGNRJMP2014} one can show: 
\begin{proposition}\label{prop3.1}
	The definition of the class $FIO(\chi,q,v_s)$ is independent of the
	window function $g\in\cS(\rd)$.
\end{proposition}
\begin{remark}
	(i) For $q=1$ the original definition of $FIO(\chi,v_s)$ in \cite{CGNRJMP2014} was formulated in terms of a function $H\in L^1_{v_s}(\rdd)$ instead of the more restrictive condition $H\in W(C,L^1_{v_s})(\rdd)$. Though, it turns out that the two definitions are equivalent, see \cite[Proposition 3.1]{CGNRJMP2014}. \\
	(ii) Similarly  to $q=1$, one could consider   the algebra of $FIO(\chi,\infty ,v_s)$, $s>2d$ such that 
		\begin{equation}\label{asteriscoinfty}
	|\langle T \pi(z) g,\pi(w)g\rangle|\leq \la w-\chi z\ra^{-s},\qquad \forall w,z\in\rdd.
	\end{equation}
	We shall not treat this case explicitly, but we remark that it enjoys similar  properties to those we are going to establish for the cases above.
\end{remark}

\begin{theorem}\label{cara}
	Consider $T$ a continuous linear operator $\cS(\rd)\to\cS'(\rd)$,
	$\chi\in Sp(d,\bR)$, $0< q\leq 1$, $s\geq 0$. Let $\mathcal{G}(g,\Lambda)$ be a Gabor frame with
	$g\in\cS(\rd)$. Then the following properties are
	equivalent:\par
	 {(i)} there exists a function $H\in W(C,L^q_{v_s})(\rdd)$, 
	such that the kernel of $T$ with respect to \tfs s satisfies the decay
condition \eqref{asterisco};
\par {(ii)} there exists a sequence $h\in \ell^q_{v_s}(\Lambda)$, 
	such that 
	\begin{equation}\label{unobis2}
	|\langle T \pi(\lambda) g,\pi(\mu)g\rangle|\leq h( \mu-\chi(\lambda)),\qquad \forall \lambda,\mu\in \Lambda.
	\end{equation}
\end{theorem}
\begin{proof}
	It is a straightforward modification of the proof \cite[Theorem 3.1]{CGNRJMPA}.
\end{proof}

We list a series of issues which follow by easy modifications of the earlier results contained in \cite{CGNRJMPA,CGNRJMP2014}, for a detailed proof we refer to \cite{CG2022}.
\begin{theorem}\label{listprop}
(i) \textit{Boundedness}.	Fix $\chi\in\Spnr$, $0< q\leq 1$, $s\geq 0$, $m\in \cM_{v_s}$ and let $T$ be
	generalized metaplectic operator in $FIO(\chi,q,v_s)$. Then $T$ is
	bounded from $M^p_m(\rd)$ to $M^p_{m\circ\chi^{-1}}(\rd)$, $q\leq p\leq
	\infty$. \\
	(ii) \textit{Algebra property}. Let $\chi_i\in\Spnr$, $s\geq0$ and $T_i\in FIO(\chi_i,q, v_s)$, $i=1,2$. Then $T_1T_2\in FIO(\chi_1\chi_2,q,v_s)$.\\
\end{theorem}

For the invertibility property, the algebra cases corresponding to the spaces of sequences  $\ell^1_{v_s}$ where treated in \cite{charly06} and \cite{GR} (see also earlier references therein). We extend those arguments to the quasi-Banach setting  as follows. 
\begin{definition} Consider $\cB=\ell^q_{v_s}(\Lambda)$, $0< q\leq 1$, $s\geq0$. Let $A$ be a matrix on $\Lambda$ with entries $a_{\lambda,\mu}$, for $\lambda,\mu\in \Lambda$, and let $d_A$ be the sequence with entries $d_A(\mu)$ defined by
\begin{equation}\label{3.2}
d_A(\mu)=\sup_{\lambda\in\Lambda}|a_{\lambda,\lambda-\mu}|.
\end{equation} 
We say that the matrix $A$ belongs to $\cC_\cB$ if $d_A$ belongs to $\cB$. The (quasi-)norm in $\cC_\cB$ is given by
$$\|A\|_{\cC_\cB}=\|d\|_{\cB}.$$
\end{definition}
The value $ d_A(\mu)$ is the supremum of the entries in the $\mu-th$ diagonal of $A$, thus the $\cC_\cB$-norm describes a form of the off-diagonal decay of a matrix.
\begin{theorem}\label{c-i}
Consider the (quasi-)algebra $\cB$ above. Then the following are equivalent:\\
(i) $\cB$ is inverse-closed in $B(\ell^2)$.\\
(ii) $\cC_\cB$ is inverse-closed in $B(\ell^2)$.\\
(iii) The spectrum $\widehat{\cB}\simeq \bT^d$.	
\end{theorem}
\begin{proof}
 The algebra case is already proved in \cite{GR}. The quasi-algebra case follows by a similar pattern, since, for $0<q<1$, it is easy to check that $\ell^q_{v_s}(\Lambda)$ is a solid convolution quasi-algebra of sequences.
\end{proof}

As a consequence, we can state:
\begin{theorem}\label{6.8}
	The class of Weyl operators with symbols in $M^{\infty, q}_{1\otimes v_s}(\rdd)$, $0< q\leq 1$, is inverse-closed in $B(L^2(\rd))$. In other words, if $\sigma\in M^{\infty, q}_{1\otimes v_s}(\rdd)$ and $Op_w(\sigma)$ is invertible on $\lrd$, then $(Op_w(\sigma))^{-1}=Op_w(b)$ for some $b\in M^{\infty, q}_{1\otimes v_s}(\rdd)$. 
\end{theorem}
\begin{proof} It follows the pattern of Theorem 5.5 in \cite{GR}, using Theorem \ref{c-i} in place of the corresponding Theorem 3.5 in the above-mentioned paper.
	\end{proof}
\begin{theorem}[Invertibility in the class $FIO(\chi,q,v_s)$]\label{inverse}
Consider $T\in FIO(\chi,q,v_s),$ such that $T$ is invertible on $L^2(\rd)$,
then $T^{-1} \in FIO(\chi^{-1},q,v_s)$.
\end{theorem}
\begin{proof} The pattern is similar to Theorem 3.7 in \cite{CGNRJMPA}. We detail the differences. We first show that the adjoint operator $T^\ast$ 
belongs to the class $FIO(\chi \inv, q,v_s)$. By Definition \ref{def1.1}:
\begin{align*}|\langle T^*\pi(z)g,\pi(w) g \rangle|&=|\langle \pi(z)g,T(\pi(w) g) \rangle|=|\langle T(\pi(w) g),\pi(z)g \rangle|\\
&\leq H( z-\chi(w))=\cI (H\circ\chi)(w-\chi^{-1}z).
\end{align*}
It is easy to check that $\cI (H\circ\chi)\in W(C,L^q_{v_s})$ for $H\in W(C,L^q_{v_s})$, since $v_s\circ \chi^{-1}\asymp v_s$, and the claim follows.
Hence, by Theorem \ref{listprop} $(ii)$, the operator $P:=T^\ast T$ is in
$FIO(\mathrm{Id},q,{v_s})$ and satisfies the estimate \eqref{unobis2}, that is: $$|\langle P\pi
(\lambda )g, \pi (\mu )g\rangle | \leq h (\lambda - \mu
), \quad\forall \lambda , \mu \in \Lambda, $$
 and a suitable sequence $h\in \ell^q_{v_s}(\Lambda)$.
 The characterization for pseudodifferential operators  in
 Theorem 3.2 \cite{BC2021} says that $P$ is a Weyl operator  $P=Op_w(\sigma)$ with a symbol $\sigma$ in
 $M^{\infty,q}_{1\otimes v_s}(\rdd)$. Since $T$ and therefore
 $T^\ast$ are invertible on $L^2(\rd)$, $P$ is also invertible on
 $L^2(\rd)$. Now we apply Theorem~\ref{6.8} and conclude that the inverse
 $P^{-1}=Op_w(\tau)$ is a  Weyl 
 operator with symbol in $\tau\in M^{\infty,q}_{1\otimes v_s}(\rdd)$.
 Hence $P^{-1}$ is in $FIO(\mathrm{Id},q,v_s)$. Eventually, using the
 algebra property of Theorem \ref{listprop} $(ii)$, we obtain that $T^{-1}=P^{-1}
 T^\ast$ is in $FIO(\chi \inv ,q,v_s)$.
\end{proof}

\begin{theorem}\label{pseudomu}
	Fix $0< q\leq 1$, $\chi \in \Spnr $. A linear continuous operator $T:
	\cS(\rd)\to\cS'(\rd)$ is in  $FIO(\chi,q,v_s)$ if and only if there
	exist symbols $\sigma_1, \sigma_2 \in
	M^{\infty,q}_{1\otimes v_s}(\rdd)$, 
	such that
	\begin{equation}\label{pseudomu1}
	T=Op_w(\sigma_1)\mu(\chi)\quad \mbox{and}\quad
	T=\mu(\chi)Op_w(\sigma_2).
	\end{equation}
	The symbols $\sigma _1$ and $\sigma _2$ are related by
	\begin{equation}\label{hormander}\sigma_2=\sigma_1\circ\chi.\end{equation}
\end{theorem}
\begin{proof}
It follows the same pattern of the proof of \cite[Theorem 3.8]{CGNRJMP2014}. The main tool is the characterization in Theorem 3.2 of \cite{BC2021} which extends Theorem 4.6 in \cite{GR} to the case $0<q<1$. We recall the main steps for the benefit of the reader. \par 
Assume $T\in FIO(\chi,q,v_s)$ and fix $g\in\cS(\rd)$. We first prove the factorization $T=\sigma_1^w\mu(\chi)$.
For every $\chi\in\Spnr$, the kernel of $\mu (\chi )$ with respect to
\tfs s can be written as
$$|\langle \mu(\chi) \pi(z)g,\pi(w)g\rangle|=|V_{g}\big(\mu(\chi)g\big)\big(w-\chi z\big)|.$$
Since both  $g\in \cS (\rd )$ and
$\mu(\chi)g\in\cS(\rd)$, we have
$V_{g}(\mu(\chi)g)\in\cS(\rdd)$ (see e.g., \cite{Elena-book}).
Consequently, we have found a function $H=|V_{g}\big(\mu(\chi)g\big)|\in \cS(\rdd)\subset W(C, L^q_{v_s})$
such that
\begin{equation}
|\langle \mu(\chi) \pi(z)g,\pi(w)g\rangle| \leq H(w-\chi z) \, \quad w,z\in \rdd . \label{stimam}
\end{equation}
Since $\mu(\chi)^{-1}=\mu(\chi^{-1})$ is in $FIO( \chi \inv ,q,v_s)$ by
Theorem \ref{inverse}, the algebra property of Theorem~\ref{listprop} $(ii)$
implies that
$T\mu(\chi^{-1})\in FIO(\mathrm{Id},q,v_s)$. Now Theorem 3.2 in \cite{BC2021}
implies the existence of a symbol
$\sigma _1 \in M^{\infty,q}_{1\otimes v_s}(\rdd)$, such that
$T\mu(\chi)^{-1}=Op_w(\sigma_1)$, as claimed. The rest goes exactly as in \cite[Theorem 3.8]{CGNRJMP2014}.
\end{proof}
\section{Applications to Schr\"odinger Equations}
The theory developed in the previous sections finds a natural application in quantum mechanics. In particular, we focus on the Cauchy problems for Schr\"odinger equations announced in the Introduction, cf. \eqref{C1}, with Hamiltonian of the form \eqref{C1bis}:
$$
H=Op_w(a)+ Op_w(\sigma),
$$
where $Op_w(a)$ is the Weyl quantization of a real homogeneous quadratic polynomial on
$\rdd$ and $Op_w(\sigma)$ is a pseudodifferential operator with a
symbol $\sigma\in M^{\infty,q}_{1\otimes v_s}(\rdd)$. 

Proposition \ref{lemmaGrochenig} (see also Theorem \ref{listprop} $(i)$ with $\chi=Id$ or \cite[Theorem 3.1]{ToftquasiBanach2017}) gives
\begin{corollary}
If $\sigma\in M^{\infty,q}_{1\otimes v_s}(\rdd)$, $s\geq0$, $0<q\leq1$, then the operator $Op_w(\sigma)$ is bounded on all \modsp s $M^p_{v_s}(\rd)$, for $q\leq p\leq \infty$. In particular, $Op_{w}(\sigma)$ is bounded
on $\lrd $.
\end{corollary}

This implies that the operator  $H$ in \eqref{C1bis} is a bounded perturbation of the generator $H_0 = Op_w(a)$ of a unitary
group (cf. \cite{RS75}), and $H$ is the generator of a well-defined
(semi-)group.


\begin{theorem}\label{teofinal}
	Let $H$ be the Hamiltonian in \eqref{C1bis} with
	homogeneous polynomial $a$ and $\sigma\in
	M^{\infty,q}_{1\otimes v_s}(\rdd )$, $0<q\leq 1$, $s\geq0$. Let $U (t) = e^{itH}$ be the corresponding
	propagator. Then $U(t)$ is a generalized metaplectic operator for each
	$t \in\bR$.
Namely, the solution of the homogenous problem $iu_t + Op_w(a) u = 0$ is given by a
	metaplectic operator $\mu(
	\mathcal{\chi}_t)$ in \eqref{e7}, and $e^{itH}$
	is of the from
	\[
	e^{itH} = \mu(\chi_t)Op_w(b_t)
	\]
	for some symbol $b_t\in M^{\infty,q}_{1\otimes v_s}(\rdd)$.
\end{theorem}
\begin{proof}
	The proof of the above result was shown for $q=1$ in \cite{CGNRJMP2014} and it easily extends to any $0<q<1$. In fact, the main ingredients to use are
	the invariance of $M^{\infty,q}_{1\otimes v_s}(\rdd) $ under metaplectic operators, plus
		 the properties of that symbol class: the
		boundedness on \modsp s and the algebra  property of the corresponding Weyl operators.
\end{proof}
\begin{corollary}\label{corfinal}
	Assume $\sigma\in M^{\infty,q}_{1\otimes v_s}(\rdd)$, $0< q\leq 1$, $s\geq0$, $m\in\cM_{v_s}$. If the initial condition $u_0$ is in $M^p_m(\rd)$, with $q\leq p\leq\infty$, then $u(t,\cdot)\in M^p_{m\circ\chi^{-1}}(\rd)$ for every $t\in\bR$. In particular, if $m\circ\chi^{-1}\asymp m$ for every $\chi\in Sp(d,\bR)$ (as for $v_s$)
 the time evolution leaves $M^p_m(\rd)$ invariant: the Schr\"odinger evolution 	preserves the phase space concentration of the initial condition.
\end{corollary}
\begin{proof}[Proof of Corollary \ref{corfinal}]
	It follows from Theorem \ref{teofinal} and the representation in
	Theorem \ref{pseudomu} that $e^{itH}\in FIO(\chi,q,v_s)$, so that the
	claim is direct consequence of Theorem \ref{listprop} $(i)$.
\end{proof}

We can now study the Wigner  kernel of $e^{itH}$, namely $k(z,w)$, $w,z\in\rdd$, such that
$$ W(e^{itH}u_0)(z)=\int k(t,z,w) Wu_0(w)\,dw,$$
and possible generalizations to $\A$-Wigner transforms.

For sake of clarity, we start with a symbol $\sigma$ in the H\"{o}rmander class $S^0_{0,0}(\rdd)$, that is $\sigma \in\cC^{\infty}(\rdd)$ such that for every $\alpha\in\bN^d$ there exists a $C_\alpha>0$ for which 
$$|\partial^\alpha \sigma(z) |\leq C_\alpha,\quad \forall z\in\rdd.$$
We recall that  $S^0_{0,0}(\rdd)$ can be viewed as the intersection of  modulation spaces \cite{BC2021}
$$S^0_{0,0}(\rdd)=\bigcap_{s\geq 0} M^{\infty,q}_{1\otimes \la \cdot\ra^s}(\rdd),\quad 0<q\leq\infty.$$
Let $\cA$ be a covariant, shift-invertible matrix.  Actually, when working in the $L^2$ setting, 
the assumption of shift-invertibility will be not essential in the sequel. 
We may argue in terms of the Cohen class $Q_{\Sigma}$ in Theorem \ref{ThmCohen}:
$$W_\cA f=Wf\ast \Sigma= Q_{\Sigma}f$$
	where $\Sigma$ is related to $\cA$ by \eqref{CohenSymbol}, \eqref{covBCohen}.
	
	Let us define 
	$\Sigma_t(z)=\Sigma(\chi_t(z))$ and denote by $\cA_t$ the corresponding covariant matrix, such that $W_{\cA_t}=Q_{\Sigma_t}$, see Proposition 4.4 in \cite{CR2022} for details. Note that in the case of the standard Wigner transform we have $W=W_\cA=W_{\cA_t}$ for every $t$, since $\Sigma=\delta$.
	
	The following proposition is the Wigner counterpart for $e^{itH}$ of the almost-diagonalization in Definition \ref{def1.1}.
	\begin{proposition}\label{Y} Under the assumptions above, for $z\in\rdd$, $t\in\bR,$
		\begin{equation}\label{Yeq}
		W_\cA(e^{itH}u_0)(z)=\int_{\rdd} k_{\cA} (t,z,w) (W_{\cA_t}u_0)(w)\,dw,
		\end{equation}
	where for every $N\geq0$,
	$$k_{\cA} (t,z,w)\la w-\chi_t^{-1}(z)\ra ^{2N}$$
	is the kernel of an operator bounded on $\lrdd$.
	\end{proposition}
We need the following preliminary result, cf. Proposition 4.1, formula $(96)$ in \cite{CR2022}. To benefit the reader, we report here the proof.
	\begin{lemma}\label{YY} Under the assumptions above,
		$$
		W_\cA (\mu(\chi_t )f)(z)=W_{\cA_t}f(\chi_t^{-1}z).
		$$
	\end{lemma}
\begin{proof}
	From \cite[Proposition 1.3.7]{Elena-book} we have
	$$W(\mu(\chi_t)f)(z)=Wf(\chi_t^{-1}z),\quad f\in\cS(\rd),$$
	so that for any $\Sigma\in\cS(\rdd)$, $f\in\cS(\rd)$, 
	\begin{align*}Q_{\Sigma}(\mu(\chi_t)f)(z)&=[\Sigma\ast W(\mu(\chi_t)f)](z)\\
		&=\intrdd W(\mu(\chi_t)f)(u)\Sigma(z-u)\,du\\
			&=\intrdd Wf(\chi_t^{-1}u)\Sigma(\chi_t(\chi_t^{-1}z-\chi_t^{-1}u))\,du\\
				&=\intrdd Wf(\zeta)\Sigma(\chi_t(\chi_t^{-1}z-\zeta))\,d\zeta\\
				&=(Wf\ast \Sigma_t)(\chi_t^{-1}z)=Q_{\Sigma_t}f(\chi_t^{-1}z).
		\end{align*}
	For $\Sigma\in\cS'(\rdd)$ we may use standard approximation arguments. Since in our case $Q_{\Sigma}(\mu(\chi_t)f)=W_\cA(\mu(\chi_t)f)$ and $Q_{\Sigma_t}f(\chi_t^{-1}z)=W_{\cA_t}f(\chi_t^{-1}z)$, this concludes the proof.
\end{proof}
\begin{proof}[Proof of Proposition \ref{Y}]
	From Theorem \ref{teofinal} we have $$W_\cA (e^{it H}u_0) = W_\cA(\mu(\chi_t)Op_w(b_t)u_0) .$$
	In view of Lemma \ref{YY} 
	$$ W_\cA (e^{it H}u_0) (z)= W_{\cA_t}(Op_w(b_t)u_0)(\chi_t^{-1}z) .$$
	We now apply formula \eqref{A6} to obtain
	$$W_{\cA_t} (Op_w(b_t)u_0)= Op_w(c_t)(W_{\cA_t}u_0) $$
	where the symbol $c_t\in S^0_{0,0}(\bR^{4d})$ is given by \eqref{defc}. Summing up
\begin{equation}\label{(*)}
	W_{\cA} (e^{it H}u_0) (z)=Op_w(c_t)(W_{\cA_t}u_0)(\chi_t^{-1}z).
\end{equation}
	Writing $h(t,z,w)$ for the kernel of $Op_w(c_t)$,
	$$W_{\cA} (e^{it H}u_0) (z)=\int h(t, \chi_t^{-1}(z),w)(W_{\cA_t}u_0)(w)\,dw,$$
	that is 
	$$k_\cA(t,z,w) = h(t,\chi_t^{-1}z,w).$$
	Now, observe that for every $N\geq 0$,
	$$h_N(t,z,w)= \la z-w\ra^{2N} h(t,z,w) $$
	is the kernel of bounded operator on $\lrd$, see \cite[Lemma 5.3]{CR2021}. Hence the operator with kernel 
	$$k_\cA(t,z,w)\la \chi_t^{-1}(z)-w\ra ^{2N} =h_N(t,\chi_t^{-1}(z),w)$$
	is bounded as well.
\end{proof}
\begin{definition}\label{YYY}Fix $\cA\in Sp(2d,\bR)$ covariant and shift-invertible. For $f\in\lrd$ we define $\cW\cF_\cA (f)$, the $\cA$-Wigner wave front set of $f$, as follows. A point $z_0=(x_0,\xi_0)\not=0$ is not in $\cW\cF_\cA (f)$ if there exists a conic open neighbourhood $\Gamma_{z_0}\subset\rdd$ of $z_0$ such  that for every integer $N\geq0$ 
	$$\int_{\Gamma_{z_0}} \la z\ra^{2N} |W_\cA f(z)|^2  dz<\infty.$$
\end{definition}
In the case of the standard Wigner transform $W_\cA=W$, we write for short $\cW \cF_\cA(f)= \cW \cF(f)$. Note that $\cW\cF_\cA (f) $ is a closed cone in $\rdd\setminus\{0\}$. We have  $\cW\cF_\cA (f)=\emptyset $ if and only if $f\in\cS(\rd)$, cf. Proposition 4.7 in \cite{CR2022} and the arguments in the sequel. \par
First, we shall give the following extension of Theorem 1.6 in \cite{CR2021} concerning the $\tau$-Wigner case. 
\begin{theorem}\label{?}
	Consider $a\in S_{0,0}^0(\rdd)$. Then, for every $f\in\lrd$,
	$$\cW\cF_\cA(Op_w(a)f)\subset \cW\cF_\cA(f).$$
\end{theorem}
\begin{proof}
	Arguing exactly as in the proof of Theorem 1.6 in \cite{CR2021} and replacing $\tau$-Wigner with $\cA$-Wigner distributions, we apply the identity 
	$$W_\cA(Op_w(a)f)=Op_w(c)W_\cA f$$
	with the symbol $c$ as in \eqref{defc}  and using \eqref{A6} in Theorem \ref{5.4}  we obtain the inclusion.
\end{proof}
\begin{theorem}\label{YYYY} For $u_0\in\lrd$ we have
	$$\cW \cF_\cA(e^{it H}u_0) =\chi_t (\cW \cF_{\cA_t}(u_0)).$$
\end{theorem}

The proof follows the lines of the corresponding one in Theorem 1.6 \cite{CR2021}, basing on the preceding Proposition \ref{Y}, in particular on the identity in \eqref{(*)}, see the sketch below. For the Wigner distribution the previous result reads as follows:
\begin{corollary}\label{7.7}
	For $u_0\in\lrd$, 	$$\cW \cF(e^{it H}u_0) =\chi_t (\cW \cF (u_0)).$$
\end{corollary}
\begin{proof}[Proof of Theorem \ref{YYYY}]
	Fix $t\in\bR$, $z_0\in\rdd\setminus\{0\}$, $\Gamma_{z_0}$ small conic neighbourhood of $z_0$, $\zeta_0=\chi_t(z_0)$, $\Lambda_{\zeta_0}=\chi_t(\Gamma_{z_0})$ corresponding conic neighbourhood of $\zeta_0$. Assume $\zeta_0\notin\cW\cF_{\cA_t} (u_0)$, that is, for every $N\geq0$,
	$$\int_{\Lambda_{\zeta_0}} \la \zeta\ra^{2N} |W_{\cA_t}u_0(\zeta)|^2 d\zeta<\infty.$$
	We want to prove that $z_0\notin \cW\cF_{\cA} (e^{it H}u_0)$, that is, for every $N\geq0$,
	$$I:=\int_{\Gamma_{z_0}} \la z\ra^{2N} |W_{\cA}(e^{it H}u_0)(z)|^2 dz<\infty.$$
	By applying the basic identity \eqref{(*)} in the proof of Proposition \ref{Y} we obtain
	$$I=\int_{\Gamma_{z_0}} \la z\ra^{2N}|[Op_w(c_t)W_{\cA_t}u_0](\chi_t^{-1}z)|^2 dz$$
	and after the change of variables $z=\chi_t \zeta$, observing that $\la \chi_t \zeta\ra\asymp\la\zeta\ra$:
	$$I\asymp \int_{\Lambda_{\zeta_0}} \la\zeta\ra^{2N} |Op_w(c_t)W_{\cA_t}u_0|^2\,d\zeta.$$
	We are therefore reduced to the pseudodifferential case, cf. the preceding Theorem \ref{?}. Arguing again as in the proof of Theorem 1.6 in \cite{CR2021} and using the assumption, we obtain $\cW \cF_\cA(e^{it H}u_0)\subset \chi_t (\cW \cF_{\cA_t}(u_0))$. Similarly, one can prove the opposite inclusion.
\end{proof}
\section{Appendix. Comparison with the {H}\"ormander wave front set}
Corollary \ref{7.7} is similar to other results in the literature, concerning propagation of micro-singularities for the Schr\"odinger equation, cf. \cite{CarypisWahlberg21,CGNRJMPA,CGNRJMP2014,CNR2015,CNRadv2015,Elena-book,Craig1995,HW2005,Ito2006,Ito2009,Martinez2009,Nakamura2005,NR2015,Robbiano1999,PRW2018,P2018,W2018,Wunsch1999}. They mainly concern the global wave front set $\cW\cF_G (f)$ of H\"ormander  \cite{hormanderglobalwfs91}.  It is interesting to compare the different  microlocal contents of   $\cW\cF_\cA (f)$ and $\cW\cF_G (f)$. We recall the definition of $\cW\cF_G (f)$, following the notation and the equivalent time-frequency setting in \cite{Elena-book}. 
\begin{definition}\label{Appendix1}
	Consider $f\in\lrd$ and  $z_0\in\rdd\setminus\{0\}$. We say that  $z_0\notin \cW\cF_G (f)$  if there exists a  conic neighbourhood  $\Gamma_{z_0}\subset\rdd$ of $z_0$ such that for every integer $N\geq 0$
	\begin{equation}\label{App1}
		\int_{\Gamma_{z_0}} \la z\ra^{2N}|V_gf(z)|^2\,dz<\infty,
	\end{equation} 
where we fix $g\in\cS(\rd)\setminus\{0\}$ in the definition of the STFT $V_gf$. 
\end{definition}
As before, we assume $\cA\in Sp(2d,\bR)$ covariant and shift-invertible, then:
\begin{theorem}\label{Appendix2}
	For all $f\in\lrd$ we have
	\begin{equation}\label{App2}
	\cW\cF_G (f)\subset \cW\cF_\cA (f).
	\end{equation}
\end{theorem}
The proof requires the following preliminary issue.
\begin{lemma}\label{Appendix3}
	Fix $g\in\cS(\rd)\setminus\{0\}$ and consider $\cA\in Sp(2d,\bR)$ covariant and shift-invertible. There exists $\Psi_\cA\in\cS(\rdd)$, depending on $\cA$ and $g$ such that for every $f\in\lrd$
	\begin{equation}\label{App3}
	|V_gf|^2= \Psi_\cA\ast W_\cA f.
	\end{equation} 
\end{lemma}
\begin{proof}
	We start with the well-known identity 
	\begin{equation}\label{App4}
		|V_gf|^2= \cI W g\ast Wf,
	\end{equation}
where $\cI W g(z)=Wg(-z)$, see for example (156) in \cite{CR2021}. If $\cA$ is covariant, we have from \eqref{CohenCovWA} 
\begin{equation}\label{App5}
	W_\cA f=\Sigma_\cA\ast Wf
\end{equation}
with $\Sigma_\cA$ given by \eqref{CohenSymbol} $\Sigma_\cA(z)=\cF^{-1} (e^{-\pi i \zeta\cdot B_\cA \zeta})$. If we define 
$$\tau_\cA(z)=\cF^{-1}(e^{\pi i \zeta\cdot B_\cA \zeta})$$
we then have, for all $h\in\lrdd$, 
$$\tau_\cA\ast \Sigma_\cA\ast h=h,$$
hence from \eqref{App4} 
$$|V_gf|^2= \cI Wg\ast \tau_\cA\ast \Sigma_\cA\ast Wf=\Psi_\cA\ast W f,$$
with $\Psi_\cA= \cI Wg\ast \tau_\cA$. \par To prove that $\Psi_\cA\in\cS(\rdd)$, we observe $\cI W g\in\cS(\rdd)$, in view of the regularity property of the Wigner distribution, and $\tau_\cA\,\, \ast: \cS(\rdd)\to \cS(\rdd)$, since for every $h\in\cS(\rdd)$ we have
$$ e^{\pi i \zeta\cdot B_\cA \zeta} h(\zeta)\in\cS(\rdd).$$
This concludes the proof.
\end{proof}
\begin{proof}[Proof of Theorem \ref{Appendix2}]
The pattern is similar to the the proof of Theorem 5.5 in \cite{CR2021}, after replacing Lemma 5.4 in \cite{CR2021} with our present Lemma \ref{Appendix3}.
\end{proof}
\begin{corollary}\label{Appendix4}
Let $\cA\in Sp(2d,\bR)$ as before and $f\in\lrd$. We have $f\in\cS(\rd)$ if and only if $\cW\cF_\cA(f) =\emptyset$.
\end{corollary}
\begin{proof}
	If $f\in\cS(\rd)$, then $W_\cA f\in\cS(\rdd)$ in view of Proposition \ref{propWellDefOpA} $(ii)$. The estimates in Definition \ref{YYY} are obviously satisfied for any $z_0\in \rdd\setminus\{0\}$, hence $\cW\cF_\cA(f) =\emptyset$. In the opposite direction, assume $\cW\cF_\cA(f) =\emptyset$.  Theorem \ref{Appendix2} yields $\cW\cF_G(f)=\emptyset$ and this implies $f\in\cS(\rdd)$, cf. \cite{Elena-book}. Alternatively, one can follow the pattern of Theorem 5.4 in \cite{CR2021}, using again Lemma \ref{Appendix3}.
\end{proof}

Comparing now $\cW\cF_G(f)$ and $\cW\cF_\cA(f)$, we first observe that the definition of $\cW\cF_G(f)$ can be extended to $f\in\cS'(\rd)$ , cf. \cite{Elena-book}, whereas Definition \ref{YYY} refers to $f\in\lrd$. With some more technicalities  the definition of  $\cW\cF_\cA(f)$ can be extended to $f\in\cS'(\rd)$ as well. The substantial difference between $\cW\cF_G(f)$ and $\cW\cF_\cA(f)$ is that the inclusion in \eqref{App2} is strict in general, since  $\cW\cF_\cA(f)$ includes a ghost part depending on $\cA$, as already observed in \cite{CR2021}. \par To better understand this issue, we will use the Shubin class of symbols $H^m$, $m\in\bR$, defined by the estimates
\begin{equation}\label{shubin}
	|\partial^\alpha a(z)|\leq c_\alpha \la z\ra^{m-\alpha}, \quad z=\phas\in\rdd.
\end{equation}
Further, assume $a\in H^m_{cl}$, that is $a(z)$ has the homogeneous principal part $a_m(z)$: 
$$a_m(\lambda z)=\lambda^m a_m(z),\quad \lambda>0,$$
such that, cutting off $a_m(z)$ for small $|z|$, we have for some $\epsilon>0$, 
$a-a_m\in H^{m-\eps}$. 

Define the characteristic manifold
$$\Sigma= \{z\in\rdd, a_m(z)=0, z\not=0 \}.$$
\begin{theorem}\label{Appendix5}
Assume that $a\in H^m_{cl}$ is globally elliptic, i.e. $\Sigma=\emptyset$. Then for all $f\in\lrd$, 
$$\cW\cF_\cA(Op_w(a)f) =\cW\cF_\cA(f).$$
\end{theorem}
\begin{proof}
The inclusion $\cW\cF_\cA(Op_w(a)f) \subset\cW\cF_\cA(f)$ follows from the easy variant of Theorem \ref{?} for the class $H^m$. To obtain the opposite inclusion under the assumption of global ellipticity, we construct as in \cite{shubin} a parametrix of $Op_w(a)$. Namely, there exists a $b\in H^{-m}_{cl}$ such that 
$$ Op_w(b)Op_w(a)=I+ Op_w(r),$$
where $I$ is the identity operator and the symbol $r$ is in $\cS(\rdd)$, hence $Op_w(r): \cS'(\rd)\to\cS(\rd)$ is a regularizing operator. Therefore, 
$$f= Op_w(b)Op_w(a)f -Op_w(r)f,\quad \forall f\in\lrd,$$
with $Op_w(r)f\in\cS(\rd)$. Invoking Theorem \ref{?} 
$$\cW\cF_\cA(f)= \cW\cF_\cA(Op_w(b)Op_w(a)f)\subset \cW\cF_\cA(Op_w(a)f).$$
This completes the proof. 
\end{proof}

Theorem \ref{Appendix5} shows a similarity of $\cW\cF_\cA(f)$ with $\cW\cF_G(f)$. Though, in the non-elliptic case the classical microlocal inclusion 
$$\cW\cF_G(u)\subset \cW\cF_G(Op_w(a) u)\cup \Sigma,\quad u\in\cS'(\rd),$$
fails for $\cW\cF_\cA(u)$. Consider for simplicity the case $v=Op_w(a)u\in\cS(\rd)$, so that for the solutions of $Op_w(a)u=v$ we have
\begin{equation}\label{App6}
	\cW\cF_G(u)\subset \Sigma
\end{equation}
and test the same inclusion for $\cW\cF_\cA(u)$, $u\in\lrd$, as follows. For simplicity, we will consider only the Wigner wave front $\cW \cF (u)$ and consider in dimension $d=1$ the operator 
$$Pu=xD^2 x u=Op_w(a)u$$
where the homogeneous principal part of $a\in H^4_{cl}$ is given by 
$$a_4(x,\xi)=x^2\xi^2$$
so that $\Sigma$ is the union of the $x$ and $\xi$ axes 
\begin{equation}\label{App7}
\Sigma=\{(x,\xi)\in\bR^2, x=0 \,\mbox{or}\,\xi=0, (x,\xi)\not=(0,0)\}.
\end{equation}
We now address to the example at the end of \cite{CR2021}, where $f,g\in\lrd$ are defined such that
\begin{align}\label{App8}
Df&=i\delta-i f'\\
xg&=-i-ih,\label{App9}
\end{align}
with $f'\in\cS(\bR)$, $h\in\cS(\bR)$ and 
$$\cW \cF f=\cW \cF_G f=\{(x,\xi)\in\bR^2, x=0, \xi\not=0\},$$
$$\cW \cF g=\cW \cF_G g=\{(x,\xi)\in\bR^2, \xi=0,x\not=0\}.$$
By using \eqref{App8}, \eqref{App9}, a simple calculation shows that $Pf\in\cS(\bR), Pg\in\cS(\bR)$  and therefore for $u=f+g$ we have $Pu\in\cS(\bR)$. Then for $\Sigma$ as in \eqref{App7} we obtain $\cW\cF_G u=\Sigma$ as expected from \eqref{App6}. Instead, the non-linearity of the Wigner transform (see \cite{CR2021}) gives
$$\cW \cF u=\bR^2\setminus\{0\}.$$
To sum up, the appearance of ghost frequencies in the Wigner wave front is natural in Quantum Mechanics, but it contradicts the H\"ormander's  result for micro-ellipticity. 
\section*{Acknowledgements}
The authors have been supported by the Gruppo Nazionale per l’Analisi Matematica, la Probabilità e le loro Applicazioni (GNAMPA) of the Istituto Nazionale di Alta Matematica (INdAM).

\end{document}